\documentclass[final,onefignum,onetabnum]{siamonline171218}

\usepackage{braket,amsfonts}

\usepackage{array}
\usepackage{amsfonts,amssymb}
\usepackage{graphicx,multirow}
\usepackage{subfigure}
\usepackage{mathrsfs,enumerate,color}
\usepackage[textwidth=14cm]{geometry}

\usepackage{pgfplots}

\usepackage{hyperref}

\newsavebox{\tablebox}

\newsiamthm{thm}{Theorem}
\newsiamthm{lmm}{Lemma}
\newsiamthm{prop}{Proposition}
\newtheorem{cor}{Corollary}
\newsiamthm{myDef}{Definition}
\newtheorem{example}{Example}[section]

\newsiamremark{remark}{Remark}
\newsiamremark{hypothesis}{Hypothesis}
\crefname{hypothesis}{Hypothesis}{Hypotheses}

\usepackage{graphicx,epstopdf}
\usepackage{algorithm}
\usepackage{algorithmic}
\usepackage{amsmath}
\Crefname{ALC@unique}{Line}{Lines}

\usepackage{amsopn}

\usepackage{xspace}
\usepackage{bold-extra}

\colorlet{texcscolor}{blue!50!black}
\colorlet{texemcolor}{red!70!black}
\colorlet{texpreamble}{red!70!black}
\colorlet{codebackground}{black!25!white!25}

\newcommand{\iq}{\mathbf{i}}
\newcommand{\jq}{\mathbf{j}}
\newcommand{\kq}{\mathbf{k}}

\newcommand{\Ks}{\mathscr{K}}
\newcommand{\Ls}{\mathscr{L}}
\newcommand{\Rs}{\mathscr{R}}

\newcommand{\diag}{\operatorname{diag}}

\usepackage{epsfig,amsmath,bm,epstopdf}

\def\Aq{\boldsymbol{\mathcal{A}}}  
\def\aq{{\bf a}}

\def\bq{{\bf b}}

\newcommand{\dq}{\mathbf{d}}
\def\eq{{\bf e}}

\def\Hq{\boldsymbol{\mathcal{H}}}
\def\hq{{\bf h}}

\def\Mq{\boldsymbol{\mathcal{M}}}
\def\Nq{\boldsymbol{\mathcal{N}}}
\newcommand{\nq}{\mathbf{n}}
\def\Pq{\boldsymbol{\mathcal{P}}}

\def\Qq{\boldsymbol{\mathcal{Q}}}
\def\qq{{\bf q}}
\def\Rq{\boldsymbol{\mathcal{R}}}
\def\rqq{{\bf r}}

\def\dq{{\bf d}}
\def\rq{{\bf r}}

\newcommand{\uq}{\mathbf{u}}

\newcommand{\Vq}{\boldsymbol{\mathcal{V}}}
\newcommand{\vq}{\mathbf{v}}

\newcommand{\wq}{\mathbf{w}}
\newcommand{\Xq}{\boldsymbol{\mathcal{X}}}
\newcommand{\xq}{\mathbf{x}}

\newcommand{\yq}{\mathbf{y}}
\newcommand{\Zq}{\boldsymbol{\mathcal{Z}}}
\newcommand{\zq}{\mathbf{z}}
\newcommand{\Kq}{\boldsymbol{\mathcal{K}}}
\newcommand{\mq}{\mathbf{m}}

\def\Axq{ \boldsymbol{\mathcal{A}}{\bf x}}
\def\Avq{ \boldsymbol{\mathcal{A}}{\bf v}}
\def\AVq{ \boldsymbol{\mathcal{A}}\boldsymbol{\mathcal{V}}}

\usepackage[most]{tcolorbox}

\begin{tcbverbatimwrite}{tmp_\jobname_header.tex}
\title{Flexible Quaternion Generalized Minimal Residual Method for Ill-Posed Quaternion Inverse Problems
\thanks{This paper is supported in part by
the research grant CPG2024-00034-FST from University of Macau;
National Natural Science Foundation of China under grants 12171210, 12090011, and 11771188;
the Priority Academic Program Development Project (PAPD);
the Top-notch Academic Programs Project PPZY2015A013 of Jiangsu Higher Education Institutions.}
}

\author{Xuan Liu\thanks{Department of Mathematics, University of Macau, Macao. E-mail: yc07482@um.edu.mo}
\and Zhigang Jia\thanks{Corresponding author. School of Mathematics and Statistics and Research Institute of Mathematical Science, Jiangsu Normal University, Xuzhou 221116,
People's Republic China. Email: zhgjia@jsnu.edu.cn}
\and Xiaoqing Jin\thanks{Department of Mathematics, University of Macau, Macao. E-mail: xqjin@um.edu.mo}
}

\headers{Flexible Quaternion Generalized Minimal Residual Method}{Liu, Jia, and Jin}
\end{tcbverbatimwrite}
\input{tmp_\jobname_header.tex}

\ifpdf
\hypersetup{pdftitle={Guide to Using SIAM'S \LaTeX\ Style} }
\fi

\begin{document}
\maketitle

\begin{tcbverbatimwrite}{tmp_\jobname_abstract.tex}
\begin{abstract}
The main goal of this paper is to propose a new quaternion total variation regularization model for solving linear ill-posed quaternion inverse problems,
which arise from three-dimensional signal filtering or color image processing.
The quaternion total variation term in the model is represented by collaborative total variation regularization and approximated by a quaternion iteratively reweighted norm.
A novel flexible quaternion generalized minimal residual method is presented to quickly solve this model.
An improved convergence theory is established to obtain a sharp upper bound of the residual norm of quaternion minimal residual method (QGMRES).
The convergence theory is also presented for preconditioned QGMRES.
Numerical experiments indicate the superiority of the proposed model and algorithms over the state-of-the-art methods in terms of iteration steps, CPU time, and the quality criteria of restored color images.

\end{abstract}


\end{tcbverbatimwrite}
\input{tmp_\jobname_abstract.tex}

\section{Introduction}
In this paper we consider the large-scale linear ill-posed inverse problems
\begin{equation}\label{ill}
  \bq = \Aq\xq + \eq,
\end{equation}
where $\bq$ is an observed $N$-dimensional quaternion vector, $\Aq$ is an $N$-by-$N$ quaternion matrix,
$\xq$ is an unknown quaternion vector,
and $\eq$ represents rounding error or noise. 
Usually, $\Aq$ is ill-conditioned with ill-determined rank.
Such systems arise when discretizing inverse problems,
which are central in many applications such as three-dimensional signal processing and color image restoration
(see \cite{jnw19, jing21} and the references therein).
In practical applications, $\xq$ often denotes the unknown sharp signal (image) we wish to recover,
and $\bq$ denotes the measured signal (blurred and noisy image).
However, the quaternion inverse problem \eqref{ill} still remains an open problem due to the lack of efficient computational methods.

Recall the history of the general real large-scale linear ill-posed inverse problems $b = Ax + e$,
where $A\in \mathbb{R}^{N\times N}$ and $b,e,x\in \mathbb{R}^N$.
Many problems like this typically require the choice of a good prior that makes assumptions on the structure of the underlying signal or image we seek to estimate.
This prior often takes the form of a regularization term for an energy functional which is to be minimized \cite{Chung2011, Calvetti2005, Engl1996}.
Observing that quadratic regularization did not allow recovering sharp discontinuities, Rudin, Osher, and Fatemi proposed the total variation (TV) penalty \cite{Rudin1992} for solving this problem.
Due to the ill-conditioning of the coefficient matrix $A$ and the presence of noise in the right-hand side \cite{Vogel2002}, it is necessary to employ regularization in order to compute a meaningful approximation of $x$.
Thus, the well-known Tikhonov method can be applied \cite{Blomgren1998,Engl1996, Hansen2010, Mueller2012, Vogel2002} in general form, which computes a regularized solution
\begin{equation}\label{Tikhonov}
  x_{L,\lambda} = \arg \min _{x\in \mathbb{R}^{N}}
\|Ax-b\|_{2}^{2} + \lambda\|Lx\|_{2}^{2},
\end{equation}
where $\lambda>0$ is called a regularization parameter, which specifies the amount of smoothing
by balancing the fit-to-data term $\|Ax-b\|_{2}^{2}$ and the regularization term $\|Lx\|_{2}^{2}$,
$L\in \mathbb{R}^{M\times N}$ is called a regularization matrix,
which enforces some smoothing in $x_{L,\lambda}$ by including the penalization $\|Lx\|_{2}^{2}$ in the above objective function \cite{Kilmer2007}. 
The choice of $L$ and $\lambda$ can be crucial to obtain a good approximation of $x$.
When $L=I$, the problem is said to be in standard form.

If the goal is to preserve jumps or edges in $x$, total variation regularization is more popular \cite{Chan2005, Jensen2007}, which computes a regularized solution
\begin{equation}\label{TV}
x_{\text{TV},\lambda}=\arg\min_{x\in\mathbb{R}^{N}}\|Ax-b\|_{2}^{2}+\lambda \text{TV}(x),
\end{equation}
where the weighted term $\text{TV}(x)$ is called a total variation operator \cite{Wohlberg2007}.
In the above Tikhonov-like problem, the effect of $\lambda$TV($x$) is to produce piecewise-constant reconstructions by penalizing solutions with many steep changes in the gradient or enforcing solutions with a sparse gradient.

To solve the quaternion ill-posed inverse problems \eqref{ill}, we can generalize the above Tikhonov-like model to quaternion field.
Due to the multiplicative non-commutativity of quaternions, it is difficult to simplify and solve the model.
In the following discussion, we will establish a new quaternion iterative method to solve this problem.
The main contributions are in three aspects:

\begin{itemize}

\item A new quaternion total variation (QTV) regularization model is proposed to solve \eqref{ill}.
    By introducing a quaternion iteratively reweighted norm (QIRN) strategy,
    the challenging QTV regularization model is equivalently transformed into a quaternion Tikhonov regularization model.

\item A novel flexible quaternion generalized minimal residual method (FQGMRES) is proposed for solving the quaternion Tikhonov regularization model.
That is, the quaternion ill-posed inverse problems \eqref{ill} are solved by FQGMRES.
The convergence theory with a sharp upper bound of residual norm is proposed for both quaternion generalized minimal residual method (QGMRES) and preconditioned QGMRES.

\item The proposed algorithms are successfully applied to the three-dimensional signal filtering and color image processing.
    In color image restoration, the proposed QTV-FQGMRES improves the PSNR, SNR, and SSIM values of recovered color images by more than 20\%.
\end{itemize}

This paper is organized as follows.
In Section \ref{s:prilim}, we review quaternion matrices and collaborative total variation (CTV).
In Section \ref{s:preQGMRES}, we propose preconditioned QGMRES methods and the convergence theory.
In Section \ref{s:qctv}, we present a new quaternion total variation regularization model and solve this model.
In Section \ref{s:Experiments}, numerical examples for the three-dimensional signal filtering problem and color image restoration are tested to demonstrate that the preconditioned
QGMRES methods' performance is better than the compared methods.

\section{Preliminaries}\label{s:prilim}

In this section, we recall necessary information about quaternions, real counterpart, and CTV regularization.

\subsection{Quaternion matrices}\label{s:quaternion}
Let $\mathbb{Q}$ denote the quaternion skew-field \cite{zhf97} and $\mathbb{Q}^{m\times n}$ the set of $m\times n$ matrices on $\mathbb{Q}$.
A quaternion $\qq$ consists of four real numbers and three imaginary units, defined as follows:
$\qq=q_0+q_1\iq+q_2\jq+q_3\kq$,
with $q_0,q_1,q_2,q_3\in \mathbb{R}$ and $\iq, \jq, \kq$ satisfying $\iq^2=\jq^2=\kq^2=\iq\jq\kq=-1$.
If $q_0 = 0$, $\qq$ is a pure quaternion.
Notice that $\mathbb{Q}$ is an associative but multiplicative non-commutative algebra of rank four over $\mathbb{R}$ \cite{Hamilton66, rodl14}.
The conjugate of $\qq$ is $\overline{\qq}=q_0-q_1\iq-q_2\jq-q_3\kq$ and
the quaternion norm $|\qq|$ is defined by
$
|\qq|^2=\overline{\qq}\qq=q_0^2+q_1^2+q_2^2+q_3^2.
$
Every nonzero quaternion is invertible,
and the unique inverse is given by $1/\qq=\overline{\qq}/|\qq|^2$.
A quaternion matrix $\Aq\in \mathbb{Q}^{m\times n}$ is represented by
$
\Aq=A_0+A_1\iq+A_2\jq+A_3\kq,
$
where $A_0, A_1,A_2,A_3 \in \mathbb{R}^{m\times n}$.
Moreover, if $A_0$ is zero, then $\Aq$ is called a purely imaginary quaternion matrix.
The conjugate transpose of $\Aq$ is defined as
$
\Aq^*=A_0^T-A_1^T\iq-A_2^T\jq-A_3^T\kq.
$
Similarly with the complex case,
a quaternion matrix $\Aq\in\mathbb{Q}^{n\times n}$ is Hermitian (skew-Hermitian) if $\Aq^*=\Aq$ ($\Aq^*=-\Aq$).
A quaternion matrix $\Aq\in\mathbb{Q}^{n\times n}$ is unitary if $\Aq^*\Aq=\Aq\Aq^*=I_n$,
where $I_n$ denotes the $n \times n$ identity matrix.
A color image with the spatial resolution of $m\times n$ pixels can be represent by an $m\times n$ pure quaternion matrix \cite{jns19nla},
i.e., $ A_0=0,~ A_1=[r_{ij}],~ A_2=[g_{ij}],~ A_3=[b_{ij}]\in\mathbb{R}^{m\times n} $, and $ r_{ij}$, $ g_{ij}$ and $b_{ij}$ are respectively the red, green and blue pixel values at the location $(i,j)$ in this color image.

In the following, we review some basic definitions of quaternion matrices \cite{jing21}.
 
\begin{myDef}[Inner product]\label{d:innerproduct}
The inner product of two quaternion vectors, $\wq=[\wq_{i}],~\vq=[\vq_{i}]\in\mathbb{Q}^{n}$,  is defined as
$
\langle\wq, \vq\rangle:=\sum\limits_{i=1}^n\vq_i^*\wq_i.
$
The inner product of two quaternion matrices,
$\Mq=[\mq_{ij}],~\Nq=[\nq_{ij}]\in\mathbb{Q}^{m\times n}$,
is defined as $
\langle\Mq, \Nq\rangle:={\rm{trace}}(\Nq^*\Mq)=\sum\limits_{j=1}^n\sum\limits_{i=1}^m\nq_{ji}^*\mq_{ij}.
$
\end{myDef}
\begin{myDef}[Norm]
The $p$-norm of the quaternion vector, $\vq=[\vq_{i}]\in\mathbb{Q}^{n}$,
is defined as
$
\|\vq\|_p=\big(\sum_{i=1}^n |\vq_{i}|^p\big)^{\frac{1}{p}}, \quad p\ge 1.
$
The $p$-norm and F-norm of the quaternion matrix,
$\Aq=[\aq_{ij}]\in\mathbb{Q}^{m\times n}$,
are respectively defined as
$\|\Aq\|_p=\max_{\xq\in\mathbb{Q}^{n}/\{0\}}\frac{\|\Aq\xq\|_p}{\|\xq\|_p}$,
$p\ge 1,$
and
$ \|\Aq\|_F=\big(\sum_{j=1}^n\sum_{i=1}^m |\aq_{ij}|^2\big)^{\frac{1}{2}}.
$
\end{myDef}

Note that the $n$-dimensional quaternion vector space  $\mathbb{Q}^n$ is a right quaternionic Hilbert space with the inner product. All quaternion vectors in
$$
\Ks=\{ \vq_1\boldsymbol{\alpha}_1+\vq_2\boldsymbol{\alpha}_1+\cdots+\vq_{m}\boldsymbol{\alpha}_{m}~|
~\vq_j\in\mathbb{Q}^n,~\boldsymbol{\alpha}_j\in\mathbb{Q},~j=1,2,\ldots,m\}
$$
generate a subspace of $\mathbb{Q}^n$ of dimension ${\bf rank}([\vq_1,\vq_2,\ldots,\vq_{m}])$.

\subsection{Quaternion systems and real counterpart of quaternion matrices}\label{s:realcounterpart}

For the quaternion linear systems of the form
\begin{equation}\label{e:qsystems}
\Aq\xq=\bq,
\end{equation}
where $\Aq \in \mathbb{Q}^{n \times n}$ is an invertible quaternion matrix,
$\bq \in \mathbb{Q}^{n}$ is an $n$-dimensional quaternion vector,
and $\xq \in \mathbb{Q}^n$ is an unknown quaternion vector.
To solving this system,
one faces the difficulty is the fast implementation of quaternion operations,
including the efficient storage and exchange of their four parts (one real part and three imaginary parts).
The recently proposed structure-preserving algorithms \cite{jwl13, jwzc18, wei2018} present a new idea to overcome multiplicative non-commutativity of quaternions and the dimensional expansion problem of the widely used real presentation method of quaternion matrices.
These structure-preserving methods are designed to keep the algebraic symmetries or properties of continuous or discrete equations in the solving process,
with the aim of accurate, stable and efficient calculation.
We consider the following mapping:
\begin{equation} \label{e:RA0}
\Rs(\Aq)=\left[\begin{array}{rrrr}
A_0    &-A_1  &-A_2    &-A_3\\
A_1    &A_0   &-A_3    & A_2\\
A_2    &A_3   &A_0     &-A_1\\
A_3    &-A_2 &A_1     & A_0\\
 \end{array}
 \right],
 \end{equation}
where $\Rs(\cdot)$ is a linear homeomorphic mapping from quaternion matrices (or vectors) to their real counterpart.
The inverse mapping of $\Rs$  is defined  by $\Rs^{-1}(\Rs(\Aq))=\Aq$.
Then the quaternion linear systems \eqref{e:qsystems} can be equivalently rewritten as a real matrix equation
\begin{equation}\label{e:realME}
\Rs(\Aq)\Rs(\xq)=\Rs(\bq).
\end{equation}

\subsection{Collaborative total variation}

The class of regularizations collaborative total variation (CTV) was proposed by \cite{Duran2015}.
Define the Euclidean spaces
$X:= \mathbb{R}^{N\times C}$ and $Y:= \mathbb{R}^{N\times M\times C}$,
where $N$ is the number of pixels of a image, $M$ is the number of directional derivatives, and $C$ is the number of color channels.
We consider a color image as a two-dimensional (2D) matrix of size $N\times C$ denoted by $u=[u_{1},u_{2},\ldots,u_{C}]\in X$, with $u_{k}=[u_1^k,u_2^k, \ldots, u_N^k]^{T}\in \mathbb{R}^{N}$ for each $k\in\{1,2,\ldots,C\}$.
On the other hand, we define the linear operator $D: X\longrightarrow Y$ such that $Du\in Y$ is a three-dimensional (3D) matrix or tensor.

The general problem concerned with is
 $\min_{u\in X}\|Du\|_{\vec{b},a} +G(u),$
where $G: X \rightarrow \mathbb{R}$ is a proper, convex, lower semicontinuous (l.s.c.) functional and $\|\cdot\|_{\vec{b},a}$ is a
collaborative sparsity enforcing norm penalizing the gradient of the color image to be detailed later.
Let $col$, $der$, and $pix$ be abbreviations for color, derivative, and pixel, respectively.
For $u\in\mathbb{R}^{N\times C}$,
we can take the $\ell^{p}$ norm to one of the dimensions of the initial 2D structure,
then the $\ell^{r}$ norm along one of the dimensions of the remaining real 2D matrix.
We obtain the $\ell^{p,r}(col,pix)$ norm
$
  \|u\|_{p,r}:=\Big(\sum_{i=1}^{N}\big(\sum_{k=1}^{C}|u_i^k|^{p}\big)^{\frac{r}{p}}\Big)^{\frac{1}{r}}.
$

If we first take the $\ell^{p}$ norm to one of the dimensions of
the initial 3D structure, then the $\ell^{q}$ norm along one of the dimensions of the remaining 2D
matrix and, finally, the $\ell^{r}$ norm to the remaining vector, we obtain the $\ell^{p,q,r}(col,der,pix)$ norm
\begin{equation}\label{l_pqr}
  \|Du\|_{p,q,r}:=\Bigg(\sum_{i=1}^{N}\Big(\sum_{j=1}^{M}\big(\sum_{k=1}^{C}|D_{j}u_{i}^{k}|^{p}\big)^{\frac{q}{p}}\Big)^{\frac{r}{q}}\Bigg)^{\frac{1}{r}}.
\end{equation}

\section{Preconditioned Quaternion Generalized Minimal Residual Method}\label{s:preQGMRES}
In this section, we improve convergence rate of QGMRES and
propose the preconditioned QGMRES algorithm for solving large-scale quaternion linear systems as in \eqref{e:qsystems}. We also present a sharp upper bound of residual norm for preconditioned QGMRES.

We begin with presenting a brief review of the quaternion GMRES \cite{jing21}.
Let $\xq_0\in \mathbb{Q}^{n}$ represent an arbitrary initial guess to the solution of  \eqref{e:qsystems}, and let $\rqq_0=\bq-\Aq\xq_0$ be the associated residual vector.
We always assume that $\rqq_0\neq 0$ in the following discussion.
The algorithm QGMRES is shown in Algorithm \ref{code:QGMRES}.

\begin{algorithm}[H]
\setcounter{algorithm}{0}
\caption{QGMRES \cite{jing21}}
\label{code:QGMRES}
\begin{algorithmic}[1]
\STATE Compute $\rqq_0=\bq-\Axq_0$, $\beta:=\|\rqq_0\|_2\ne 0$, and $\vq_1:=\rqq_0/\beta$
\FOR{ $j=1,2,\ldots,m$}\label{gmresq2}
\STATE Compute $\boldsymbol{\omega}_j:=\Avq_j$\label{gmresq3}
\FOR{ $i=1,2,\ldots,j$}
\STATE  $\hq_{ij}=\langle\boldsymbol{\omega}_j,\vq_i\rangle$;
\STATE  $\boldsymbol{\omega}_j:=\boldsymbol{\omega}_j-\hq_{ij}\vq_i$
\ENDFOR
\STATE $\hq_{j+1,j}=\|\boldsymbol{\omega}_j\|_2$
\STATE {If $\hq_{j+1,j}=0$ set $m:=j$ and go to \ref{gmresq13}}
\STATE $\vq_{j+1}=\boldsymbol{\omega}_j/\hq_{j+1,j}$ \label{gmresq11}
\ENDFOR\label{gmresq12}
\STATE Define $(m+1)\times m$ quaternion Hessenberg matrix $\bar{\Hq}_m=[\hq_{ij}]_{1\leq i\leq m+1, 1\leq j\leq m}$.\\[-0.5pt]\label{gmresq13}
\STATE Compute $\yq_m$ the minimizer of $\|\beta \textbf{e}_1-\bar{\Hq}_m\yq\|_2$
and $\xq_m=\xq_0+\textbf{V}_m\yq_m$.\label{gmresq14}
\end{algorithmic}
\end{algorithm}

Introduce the quaternion Krylov subspaces
\begin{equation}\label{e:KAr0}
\Ks_m(\Aq,\rqq_0)=\texttt{span}\{\rqq_0,\Aq\rqq_0,\Aq^2\rqq_0,\ldots,\Aq^{m-1}\rqq_0\}.
\end{equation}
with any positive integer $m$.
In the following, we denote by $L_m(\Aq,\rqq_0)$ the (right-hand side) linear combination of $\rqq_0$, $\Aq\rqq_0$, $\ldots$, $\Aq^{m-1}\rqq_0$, i.e.,
\begin{equation}\label{e:LAv}
L_m(\Aq,\rqq_0):=\rqq_0\boldsymbol{\alpha}_0+\Aq\rqq_0\boldsymbol{\alpha}_1+\cdots+\Aq^{m-1}\rqq_0\boldsymbol{\alpha}_{m-1},
\end{equation}
where $\boldsymbol{\alpha}_0,\boldsymbol{\alpha}_1,\cdots,\boldsymbol{\alpha}_{m-1}$ are quaternion scalars.
Define the Krylov subspace matrix
\begin{equation*}
\Kq:=[\rqq_0,\Aq\rqq_0, \Aq^2\rqq_0, \ldots,\Aq^{m-1}\rqq_0].
\end{equation*}
Then the dimension of $\Ks_m(\Aq,\rqq_0)$ is exactly the rank of $\Kq$.
It is expected to increase as $m$ and then arrive at the maximal value, called the grade of $\rqq_0$ with respect to $\Aq$ which proposed by \cite{jing21}.

\begin{myDef}[Grade]\label{d:grade}
The \emph{grade} of $\vq$ with respect to $\Aq$ is the smallest positive integer, $\mu$, such that $\vq,\Avq,\ldots,\Aq^{\mu-1}\vq$ are linearly independent but $\vq, \Avq, \ldots, \Aq^{\mu-1}\vq$, $\Aq^{\mu}\vq$ are linearly dependent.
\end{myDef}

Based on the real structure-preserving transformations proposed in \cite{jwl13,jwzc18},
the structure-preserving property of Algorithm \ref{code:QGMRES} is shown in the process of generating the quaternion upper Hessenberg matrix $\bar{\Hq}_m\in \mathbb{Q}^{(m+1)\times m}$,
i.e., quaternion Arnoldi method.
The quaternion Arnoldi process can be used to construct a basis for this Krylov subspace \eqref{e:KAr0},
which leads to the associated decomposition
\begin{equation}\label{e:AV=VH}
\AVq_m = \Vq_{m+1}\bar{\Hq}_m,
\end{equation}
where $\Vq_{m+1}:=[\vq_1,\vq_2,\ldots,\vq_{m+1}]\in \mathbb{Q}^{n\times (m+1)}$ has orthonormal columns that span the quaternion Krylov subspace $\Ks_m(\Aq,\rqq_0)$.
Let $\Hq_m:=H_0+H_1\iq+H_2\jq+H_3\kq$ be the matrix obtained from $\bar{\Hq}_m$ by deleting its last row.
Then we have
$
\Vq_m^*\AVq_m = \Hq_m.
$
Note that it is inexpensive to compute the minimizer $\yq_m$ in line \ref{gmresq14} of Algorithm \ref{code:QGMRES} of the upper Hessenberg quaternion least-squares problem (HQLS) when $m$ is small.
A method for solving this problem is also presented in \cite{jing21}.

\subsection{Improved convergence theory of QGMRES}\label{s:improvedconvergence}

Our goal is to improve the upper bound of the residual norm of QGMRES and present a new global convergence result \cite{sasc86}.
To do this we propose the definition of quaternion polynomials for the first time.

\begin{myDef}[Quaternion polynomial]
A quaternion polynomial is defined as
\begin{equation}\label{e:qpoly}
p(\boldsymbol{\lambda})=\boldsymbol{\alpha}_0+\boldsymbol{\lambda}\boldsymbol{\alpha}_1+\cdots+\boldsymbol{\lambda}^{n}\boldsymbol{\alpha}_{n},
\end{equation}
where $\boldsymbol{\lambda}\in \mathbb{Q}$ and $\boldsymbol{\alpha}_j$ $(j=1,2,\ldots,n)$ are quaternion coefficients. The set of all these quaternion polynomials of degree $n$ is denoted by $\mathbb{P}_n^q$.
\end{myDef}

Recall that $\boldsymbol{\alpha}\in \mathbb{Q}$ is similar to $\boldsymbol{\beta}\in \mathbb{Q}$ if there exists a nonzero quaternion $\boldsymbol{s}$ such that $\boldsymbol{\alpha}=\boldsymbol{s}\boldsymbol{\beta}\boldsymbol{s}^{-1}$.
Denote the set of similar quaternions of $\boldsymbol{\alpha}$ by $[\boldsymbol{\alpha}]_s$ and then we define similar quaternion polynomials.

\begin{myDef}
Suppose $p(\boldsymbol{\lambda})=\boldsymbol{\alpha}_0+\boldsymbol{\lambda}\boldsymbol{\alpha}_1+\cdots+\boldsymbol{\lambda}^{n}\boldsymbol{\alpha}_{n}$ and $q(\boldsymbol{\lambda})=\boldsymbol{\beta}_0+\boldsymbol{\lambda}\boldsymbol{\beta}_1+\cdots+\boldsymbol{\lambda}^{n}\boldsymbol{\beta}_{n}$ are two quaternion polynomials.
If $\boldsymbol{\alpha}_j=\boldsymbol{s}_j\boldsymbol{\beta}_j\boldsymbol{s}_j^{-1}$
for $\boldsymbol{s}_j\in \mathbb{Q}/{0}$, $j=1,2,\ldots,n$,
then we say that $p(\boldsymbol{\lambda})$ is similar to $q(\boldsymbol{\lambda})$.
The set of similar quaternion polynomials of $p(\boldsymbol{\lambda})$ is denoted by $[p(\boldsymbol{\lambda})]_s$.
\end{myDef}

If all $\boldsymbol{s}_j$ are equal to $\boldsymbol{s}$, then
\begin{align*}
q(\boldsymbol{\lambda})&=
\boldsymbol{s}^{-1}\boldsymbol{\alpha}_0\boldsymbol{s}+\boldsymbol{\lambda}\boldsymbol{s}^{-1}\boldsymbol{\alpha}_1\boldsymbol{s}+\cdots+\boldsymbol{\lambda}^{n}\boldsymbol{s}^{-1}\boldsymbol{\alpha}_{n}\boldsymbol{s}\ =\boldsymbol{s}^{-1}p(\boldsymbol{s}\boldsymbol{\lambda}\boldsymbol{s}^{-1})\boldsymbol{s}.
\end{align*}
That means if $\boldsymbol{\gamma}\in [\boldsymbol{\alpha}]_s$ then $p(\boldsymbol{\gamma})=\boldsymbol{s}q(\boldsymbol{\lambda})\boldsymbol{s}^{-1}$.

Since quaternions are multiplicatively noncommutative, the above-mentioned quaternion polynomial is a right-hand side linear combination of $1,\boldsymbol{\lambda},\ldots,\boldsymbol{\lambda}^{n}$.
These quaternion polynomials will be used to derive the convergence theory of QGMRES.
Next, we always assume that $\Aq\in \mathbb{Q}^{n \times n}$ is a diagonalizable quaternion matrix and $\Aq=\Xq\boldsymbol{\Lambda} \Xq^{-1}$,
where $\boldsymbol{\Lambda}=\diag(\boldsymbol{\lambda}_1,\boldsymbol{\lambda}_2,\ldots,\boldsymbol{\lambda}_n)$.
By the definitions of quaternion polynomial and similar quaternion polynomials, we give the concept of quaternion minimal polynomial.

\begin{myDef}[Quaternion minimal polynomial]
Quaternion polynomial $p(\boldsymbol{\lambda})$ is called the quaternion minimal polynomial of $\Aq$ if $p(\boldsymbol{\lambda})$ is the nonzero quaternion polynomial of lowest degree such that $q(\boldsymbol{\lambda}_i)=0$ holds for any $q(\boldsymbol{\lambda}) \in [p(\boldsymbol{\lambda})]_s$ and any eigenvalue $\boldsymbol{\lambda}_i$ of $\Aq$.

\end{myDef}

\begin{thm}
Assume that $\Aq\in \mathbb{Q}^{n \times n}$ is a diagonalizable quaternion matrix.
Let $\Aq=\Xq\boldsymbol{\Lambda} \Xq^{-1} $and $\Xq^{- 1}\rq_0:=[\xq_1,\xq_2,\ldots,\xq_n]^T$,
where
$
\boldsymbol{\Lambda}=\diag(\boldsymbol{\lambda}_1,\boldsymbol{\lambda}_2,\ldots,\boldsymbol{\lambda}_n)
$
is the diagonal matrix of eigenvalues.
Define
\begin{equation*}
  \epsilon^{(m)}=
  \min_{p\in \mathbb{P}_m^q,~p(0)=1}
  \max_{i=1,\ldots,n}
    |p(\hat{\boldsymbol{\lambda}}_i)|
    =\min_{p\in \mathbb{P}_m^q,~p(0)=1}
  \max_{i=1,\ldots,n}
    |q( \boldsymbol{\lambda}_i)|.
\end{equation*}
where
$$
p(\hat{\boldsymbol{\lambda}}_i)= 1 -\hat{\boldsymbol{\lambda}}_i \boldsymbol{\alpha}_0-\hat{\boldsymbol{\lambda}}_i^2 \boldsymbol{\alpha}_1-\cdots-\hat{\boldsymbol{\lambda}}_i^{m}\boldsymbol{\alpha}_{m-1},\quad
\hat{\boldsymbol{\lambda}}_i:=\xq_i^{-1}\boldsymbol{\lambda}_i \xq_i,
$$
$$
q( \boldsymbol{\lambda} _i)
= \xq_i(1 - \xq_i^{-1}\boldsymbol{\lambda}_i \xq_i\boldsymbol{\alpha}_0-\xq_i^{-1}\boldsymbol{\lambda}_i^2 \xq_i\boldsymbol{\alpha}_1-\cdots-\xq_i^{-1}\boldsymbol{\lambda}_i^{m}\xq_i\boldsymbol{\alpha}_{m-1})\xq_i^{-1}.
$$
Then, the residual norm is achieved by the $m$th step of QGMRES satisfies the inequality
\begin{equation*}
  \|\rqq_m\|_2\leq\kappa_2(\Xq)\epsilon^{(m)}\|\rq_0\|_2,
\end{equation*}
where $\kappa_2(\Xq):=\|\Xq\|_2\|\Xq^{-1}\|_2$.
\end{thm}

\begin{proof}
Let $L_{j+1}(\Aq,\rqq_0)$ satisfy the constraint $L_{j+1}(0,\rqq_0)=\rqq_0~(\boldsymbol{\alpha}_0=1)$,
and $\xq$ the vector from $\Ks_m$
to which it is associated via
$\bq-\Axq=L_{j+1}(\Aq,\rqq_0)$.
Let
$\boldsymbol{D}_j:=\diag(\dq_{j1},\dq_{j2},\ldots,\dq_{jn})$ be a diagonal quaternion matrix with
$$
\dq_{ji}=\left\{\begin{array}{ll}
\xq_i\boldsymbol{\alpha}_j\xq_i^{-1},&  \quad \xq_i\neq0  \\[0.1in]
0,  & \quad \xq_i=0.  \end{array}\right.
$$
Then by $\xq_m=\xq_0+L_m(\Aq,\rqq_0)$, we derive
\begin{align*}
&\|\bq-\Axq_m\|_2=\|\rqq_0-\Aq L_m(\Aq,\rqq_0)\|_2\\
=&\|\rqq_0-\Xq\boldsymbol{\Lambda}\boldsymbol{D}_0\Xq^{-1} \rqq_0 -\Xq\boldsymbol{\Lambda}^2 \boldsymbol{D}_1\Xq^{-1} \rqq_0-\cdots-\Xq\boldsymbol{\Lambda}^m \boldsymbol{D}_{m-1}\Xq^{-1} \rqq_0\|_2\\
\leq&\|\Xq\|_2\|I-\boldsymbol{\Lambda}\boldsymbol{D}_0-\boldsymbol{\Lambda}^2 \boldsymbol{D}_1-\cdots-\boldsymbol{\Lambda}^m \boldsymbol{D}_{m-1}\|_2\|\Xq^{-1}\|_2\|\rqq_0\|_2.
\end{align*}

Since $\boldsymbol{\Lambda}$ and $\boldsymbol{D}_j (j=0,1,\ldots,m-1)$ are diagonal,
$$
\|I-\boldsymbol{\Lambda}\boldsymbol{D}_0-\boldsymbol{\Lambda}^2 \boldsymbol{D}_1-\cdots-\boldsymbol{\Lambda}^m \boldsymbol{D}_{m-1}\|_2
\leq\max_{i=1,\ldots,n}
    |p(\hat{\boldsymbol{\lambda}}_i)|.
$$
Now the polynomial $p$ which minimizes the right-hand side in the above inequality can be used.
This yields the desired result,
\begin{equation}
\|\bq-\Axq_m\|_2
\leq\min_{p\in \mathbb{P}_m^q,~p(0)=1}
  \max_{i=1,\ldots,n}
    |p(\hat{\boldsymbol{\lambda}}_i)|
    =\min_{p\in \mathbb{P}_m^q,~p(0)=1}
  \max_{i=1,\ldots,n}
    |q( \boldsymbol{\lambda}_i)|.
\end{equation}
\end{proof}

\subsection{Left-Preconditioned QGMRES}\label{s:leftQGMRES}
As with the preprocessing of the GMRES algorithm \cite{Bjorck1996, saad03} for real number fields,
we consider various preprocessing of quaternion GMRES.
Define the left preconditioned QGMRES is that the QGMRES algorithm applied to the quaternion system
\begin{equation}\label{e:PAx=Pb}
 \Pq^{-1}\Aq\xq=\Pq^{-1}\bq,
\end{equation}
where $\Aq \in \mathbb{Q}^{n \times n}$, $\bq, \xq \in \mathbb{Q}^{n}$, and $\Pq^{-1}\in \mathbb{Q}^{n \times n}$ is a left preconditioner.
If $\xq_0\in \mathbb{Q}^{n}$ represent an arbitrary initial guess to the solution and $\rqq_0=\bq-\Aq\xq_0$ is the residual vector of \eqref{e:qsystems}, then $\zq_0=\Pq^{-1}(\bq-\Aq\xq_0)$ is the preconditioned residual of \eqref{e:PAx=Pb}.
The straightforward application of QGMRES to \eqref{e:PAx=Pb} yields the left preconditioned QGMRES (Algorithm \ref{code:leftQGMRES}).

\begin{algorithm}
\setcounter{algorithm}{1}
\caption{QGMRES with Left Preconditioning (QGMRES$_{lp}$)}
\label{code:leftQGMRES}
\begin{algorithmic}[1]
\STATE Compute $\textbf{z}_0=\Pq^{-1}(\bq-\Aq\xq_0)$, $\beta:=\|\textbf{z}_0\|_2\ne 0$, and $\vq_1:=\textbf{z}_0/\beta$
\FOR{ $j=1,2,\ldots,m$}\label{leftgmresq2}
\STATE Compute $\boldsymbol{\omega}_j:=\Pq^{-1}\Avq_j$\label{leftgmresq3}
\FOR{ $i=1,2,\ldots,j$}
\STATE  $\hq_{ij}=\langle\boldsymbol{\omega}_j,\vq_i\rangle$;
\STATE  $\boldsymbol{\omega}_j:=\boldsymbol{\omega}_j-\hq_{ij}\vq_i$
\ENDFOR
\STATE $\hq_{j+1,j}=\|\boldsymbol{\omega}_j\|_2$
\STATE {If $\hq_{j+1,j}=0$ set $m:=j$ and go to \ref{leftgmresq13}}
\STATE $\vq_{j+1}=\boldsymbol{\omega}_j/\hq_{j+1,j}$
\ENDFOR\label{leftgmresq12}
\STATE Define $(m+1)\times m$ quaternion Hessenberg matrix $\bar{\Hq}_m=[\hq_{ij}]_{1\leq i\leq m+1, 1\leq j\leq m}$.\\[-0.5pt]\label{leftgmresq13}
\STATE Compute $\yq_m$ the minimizer of $\|\beta \textbf{e}_1-\bar{\Hq}_m\yq\|_2$
and $\xq_m=\xq_0+\textbf{V}_m\yq_m$.\label{leftgmresq14}
\end{algorithmic}
\end{algorithm}

The quaternion Arnoldi process constructs an orthogonal basis of the left preconditioned quaternion Krylov subspace
\begin{equation}\label{e:KPAr0}
\Ks_m^L:=\Ks_m(\Pq^{-1}\Aq,\zq_0)=\texttt{span}\{\zq_0,\Pq^{-1}\Aq\zq_0,(\Pq^{-1}\Aq)^2\zq_0,\ldots,(\Pq^{-1}\Aq)^{m-1}\zq_0\}.
\end{equation}
Here, $\Ks_m(\Pq^{-1}\Aq,\zq_0)=\Ks_m(\Pq^{-1}\Aq,\Pq^{-1}\rqq_0)$.

A matrix $\Pq^{-1}$ is a good preconditioner if the application of an iterative method of interest to the preconditioned linear system of equations \eqref{e:PAx=Pb} gives a faster convergence rate of the computed iterates than the application of the iterative method to the original linear system \eqref{e:qsystems}.

Now, we establish an upper bound on the convergence rate of the left preconditioned QGMRES iterates first and then present a global convergence result.
We generalize the result in \cite[Proposition 5.3]{saad03} to the quaternion skew-field.
\begin{lemma}\label{l:galerkin}
Let $\Aq$ be an arbitrary square quaternion matrix and assume that $\Ls=\Aq\Ks$.
Then a vector $\tilde{\xq}$ is the result of an (oblique) projection method onto $\Ks$ orthogonally to $\Ls$ with the starting vector $\xq_0$ if and only if 
\begin{equation*}
R(\tilde{\xq})=\min_{\xq\in \xq_0+\Ks}R(\xq),
\end{equation*}
where $R(\xq):= \|\bq-\Aq\xq\|_2$.
\end{lemma}

By taking the residuals $\Pq^{-1}\rqq_0$ of the left preconditioned QGMRES and the quaternion subspace $\Ks_m(\Pq^{-1}\Aq,\Pq^{-1}\rqq_0)$ into Lemma \ref{l:galerkin}, the following corollary can be obtained immediately.

\begin{cor}
Let $\Aq$ be an arbitrary square quaternion matrix and assume that $\Ls_m=\Aq\Ks_m^L$.
Then a vector $\tilde{\xq}$ is the result of an (oblique) projection method onto $\Ks_m^L$ orthogonally to $\Ls_m$ with the starting vector $\xq_0$ if and only if 
\begin{equation*}
R(\tilde{\xq})=\min_{\xq\in \xq_0+\Ks_m^L}R(\xq),
\end{equation*}
where $R(\xq):= \|\Pq^{-1}(\bq-\Aq\xq)\|_2$.
\end{cor}

For the minimum solution $\tilde{\xq}$, the necessary and sufficient condition for any $\yq\in \Aq\Ks_m^L$ to hold is $\yq^*\Pq^{-1}(\bq-\Aq\tilde{\xq})=0$.
In addition, this is the Petrov--Galerkin condition that defines the approximate solution.
Specific, each element of $\Ks_m^L$ is represented by the right-hand linear combination of the quaternion vectors $\zq_0,\Pq^{-1}\Aq\zq_0,(\Pq^{-1}\Aq)^2\zq_0,\cdots,(\Pq^{-1}\Aq)^{m-1}\zq_0$, rather than a production of a polynomial of $\Aq$ and the vector $\zq_0$.
Thus, for the left preconditioning option, QGMRES minimizes the residual norm
$R(\xq)$ among all vectors from the affine subspace
\begin{equation}\label{e:left_affine}
\xq_0+\Ks_m^L=\xq_0+\texttt{span}\{\zq_0,\Pq^{-1}\Aq\zq_0,(\Pq^{-1}\Aq)^2\zq_0,\ldots,(\Pq^{-1}\Aq)^{m-1}\zq_0\}
\end{equation}
in which $\zq_0$ is the initial preconditioned residual.
Then we have the following theorem.

\begin{thm}\label{left_xm}
Let $\xq_m$ be the approximate solution obtained from the $m$th step of the left-preconditioned QGMRES algorithm, and let $\zq_m=\Pq^{-1}(\bq-\Axq_m)$. Then, $\xq_m$ is of the form
\begin{equation*}
  \xq_m=\xq_0+L_m(\Pq^{-1}\Aq,\zq_0),
\end{equation*}
and
$$
\|\zq_m\|_2=\|\zq_0-\Pq^{-1}\Aq L_m(\Pq^{-1}\Aq,\zq_0)\|_2
  =\min_{j\leq m}\|\zq_0-\Pq^{-1}\Aq L_j(\Pq^{-1}\Aq,\zq_0)\|_2.
$$
\end{thm}
\begin{proof}
Recall the definition \eqref{e:LAv} that
 \begin{equation*}
 L_m(\Pq^{-1}\Aq,\zq_0)=\zq_0\boldsymbol{\alpha}_0+\Pq^{-1}\Aq\zq_0\boldsymbol{\alpha}_1+\cdots+(\Pq^{-1}\Aq)^{m-1}\zq_0\boldsymbol{\alpha}_{m-1}.
 \end{equation*}
The result follows the fact that $\xq_m$ minimizes the 2-norm of the residual in the affine subspace $\xq_0+ \Ks_m^L$ and the fact that $\Ks_m^L$ is the set of all vectors of the form $L_j(\Pq^{-1}\Aq,\zq_0)$ with $j\leq m$.
\end{proof}

This optimization condition can also be expressed in terms of the original residual vector $\rqq_0$.
Since
$$
\zq_0-\Pq^{-1}\Aq L_m(\Pq^{-1}\Aq,\zq_0)=\Pq^{-1}(\rqq_0-\Aq L_m(\Pq^{-1}\Aq,\Pq^{-1}\rqq_0)),
$$
and a simple algebraic manipulation shows that for any quaternion linear combination $L_m$,
\begin{equation}\label{e:left_affine_2}
L_m(\Pq^{-1}\Aq,\Pq^{-1}\rqq_0)=\Pq^{-1}L_m(\Aq\Pq^{-1}, \rqq_0),
\end{equation}
we obtain the relation
$$
\zq_0-\Pq^{-1}\Aq L_m(\Pq^{-1}\Aq,\zq_0)=\Pq^{-1}(\rqq_0-\Aq\Pq^{-1}L_m(\Aq\Pq^{-1}, \rqq_0)).
$$

\begin{thm}
Let $\Pq^{-1}\Aq$ be a diagonalizable quaternion matrix,
 $\Pq^{-1}\Aq=\Xq\boldsymbol{\Lambda} \Xq^{-1}$ and $\Xq^{- 1}\zq_0:=[\xq_1,\xq_2,\ldots,\xq_n]^T$, where $\boldsymbol{\Lambda}=\diag(\boldsymbol{\lambda}_1,\boldsymbol{\lambda}_2,\ldots,\boldsymbol{\lambda}_n)$
is the diagonal matrix of eigenvalues. 
Define
\begin{equation*}
  \epsilon^{(m)}=
  \min_{p\in \mathbb{P}_m^q,~p(0)=1}
  \max_{i=1,\ldots,n}
    |p(\hat{\boldsymbol{\lambda}}_i)|
    =\min_{q\in \mathbb{P}_m^q,~q(0)=1}
  \max_{i=1,\ldots,n}
    |q( \boldsymbol{\lambda}_i)|.
\end{equation*}
where
$
p(\hat{\boldsymbol{\lambda}}_i)= \boldsymbol{\alpha}_0+ \hat{\boldsymbol{\lambda}}_i\boldsymbol{\alpha}_1+ \cdots+\hat{\boldsymbol{\lambda}}_i^{m-1}\boldsymbol{\alpha}_{m-1},~~
\hat{\boldsymbol{\lambda}}_i:=\xq_i^{-1}\boldsymbol{\lambda}_i \xq_i,
$
and
$
q( \boldsymbol{\lambda} _i)
= \xq_i(\boldsymbol{\alpha}_0 + \xq_i^{-1}\boldsymbol{\lambda}_i \xq_i\boldsymbol{\alpha}_1+\cdots+\xq_i^{-1}\boldsymbol{\lambda}_i^{m-1}\xq_i\boldsymbol{\alpha}_{m-1})\xq_i^{-1}.
$
Then the residual norm achieved by the $m$th step of QGMRES  satisfies the inequality
\begin{equation*}
  \|\zq_m\|_2\leq\kappa_2(\Xq)\epsilon^{(m)}\|\zq_0\|_2,
\end{equation*}
where $\kappa_2(\Xq):=\|\Xq\|_2\|\Xq^{-1}\|_2$.

\end{thm}
\begin{proof}
Let $L_{j+1}(\Pq^{-1}\Aq,\zq_0)$ satisfy the constraint $L_{j+1}(0,\zq_0)=\zq_0~(\boldsymbol{\alpha}_0=1)$,
and $\xq$ the vector from $\Ks_m^L$
to which it is associated via
$\Pq^{-1}(\bq-\Axq)=L_{j+1}(\Pq^{-1}\Aq,\zq_0)$.
Then
\begin{align*}
\|\Pq^{-1}(\bq-\Axq)\|_2
&=\Xq\boldsymbol{D}_0\Xq^{-1}\zq_0+\Xq\boldsymbol{\Lambda} \boldsymbol{D}_1\Xq^{-1} \zq_0+\cdots+\Xq\boldsymbol{\Lambda}^j \boldsymbol{D}_j\Xq^{-1} \zq_0 \\
&\leq\|\Xq\|_2\|\boldsymbol{D}_0+\boldsymbol{\Lambda} \boldsymbol{D}_1+\cdots+\boldsymbol{\Lambda}^j \boldsymbol{D}_j\|_2\|\Xq^{-1}\|_2\|\zq_0\|_2.
\end{align*}
Since $\boldsymbol{\Lambda}$ and $\boldsymbol{D}_0, \boldsymbol{D}_1,\ldots,\boldsymbol{D}_j$ are diagonal,
$
\|\boldsymbol{D}_0+\boldsymbol{\Lambda} \boldsymbol{D}_1+\cdots+\boldsymbol{\Lambda}^j \boldsymbol{D}_j\|_2
\leq\max_{i=1,\ldots,n}|p(\hat{\boldsymbol{\lambda}}_i)|.
$
Since $\xq_m$ minimizes the residual norm over $\xq_0+\Ks_m^L$, then for any combining form $L_{j+1}(\Pq^{-1}\Aq,\rqq_0)$,
\begin{align*}
&\|\Pq^{-1}(\bq-\Axq_m)\|_2\leq\|\Pq^{-1}(\bq-\Axq)\|_2\\
&\leq\|\Xq\|_2\|\Xq^{-1}\|_2\|\zq_0\|_2\min_{p\in \mathbb{P}_m^q,~p(0)=1}\max_{i=1,\ldots,n}
|p(\hat{\boldsymbol{\lambda}}_i)|\\
&=\|\Xq\|_2\|\Xq^{-1}\|_2\|\zq_0\|_2\min_{q\in \mathbb{P}_m^q,~q(0)=1}
\max_{i=1,\ldots,n}|q( \boldsymbol{\lambda}_i)|.
\end{align*}
This yields the desired result,
$
\|\zq_m\|_2\leq\kappa_2(\Xq)\epsilon^{(m)}\|\zq_0\|_2.
$
\end{proof}

Note that in the left preconditioned QGMRES method,
the stopping criterion is not a norm of the residual $\rqq_m$ from the original system,
but a norm of the preconditioned residual $\zq_m$.
There is no better way to compute the true residuals than to compute them explicitly.
Therefore, if the stopping criterion is based on the real residuals,
then it is necessary to compute $\zq_m$ additionally,
which will cost more extra operations.

\subsection{Right-Preconditioned QGMRES}\label{s:rightQGMRES}
Consider the right preconditioned QGMRES based on solving for
\begin{equation}\label{e:APu=b}
\Aq\Pq^{-1}\uq=\bq, \quad \uq=\Pq\xq,
\end{equation}
where $\Aq \in \mathbb{Q}^{n \times n}$, $\bq, \xq \in \mathbb{Q}^{n}$, and $\Pq^{-1}\in\mathbb{Q}^{n \times n}$ is a right preconditioner.
It is necessary to distinguish between the original variable $\xq$ and the transformed variable $\uq$ related to $\uq$ by $\xq = \Pq^{-1}\uq$.
Note that the new quaternion variable $\uq$ never needs to be invoked explicitly.
Given an arbitrary initial guess $\xq_0$, then $\uq_0=\Pq\xq_0$.
The corresponding preconditioned initial residuals are
$\bq-\Aq\Pq^{-1}\uq_0=\bq-\Aq\xq_0=\rqq_0$,
i.e., the initial residual for the preconditioned system
can be computed from $\rqq_0$, which is the same as the residuals of the original system.
The algorithm right-preconditioned QGMRES is shown in Algorithm \ref{code:rightQGMRES}.

\begin{algorithm}
\setcounter{algorithm}{2}
\caption{QGMRES with Right Preconditioning (QGMRES$_{rp}$)}
\label{code:rightQGMRES}
\begin{algorithmic}[1]
\STATE Compute $\rqq_0=\bq-\Axq_0$, $\beta:=\|\rqq_0\|_2\ne 0$, and $\vq_1:=\rqq_0/\beta$
\FOR{ $j=1,2,\ldots,m$}\label{rightgmresq2}
\STATE Compute $\boldsymbol{\omega}_j:=\Aq\Pq^{-1}\vq_j$\label{rightgmresq3}
\FOR{ $i=1,2,\ldots,j$}
\STATE  $\hq_{ij}=\langle\boldsymbol{\omega}_j,\vq_i\rangle$;
\STATE  $\boldsymbol{\omega}_j:=\boldsymbol{\omega}_j-\hq_{ij}\vq_i$
\ENDFOR
\STATE $\hq_{j+1,j}=\|\boldsymbol{\omega}_j\|_2$
\STATE {If $\hq_{j+1,j}=0$ set $m:=j$ and go to \ref{rightgmresq13}}
\STATE $\vq_{j+1}=\boldsymbol{\omega}_j/\hq_{j+1,j}$ \label{rightgmresq11}
\ENDFOR\label{rightgmresq12}
\STATE Define $(m+1)\times m$ quaternion Hessenberg matrix $\bar{\Hq}_m=[\hq_{ij}]_{1\leq i\leq m+1, 1\leq j\leq m}$.\\[-0.5pt]\label{rightgmresq13}
\STATE Compute $\yq_m$ the minimizer of $\|\beta \textbf{e}_1-\bar{\Hq}_m\yq\|_2$
and $\xq_m=\xq_0+\Pq^{-1}\textbf{V}_m\yq_m$.\label{rightgmresq14}
\end{algorithmic}
\end{algorithm}

The quaternion Arnoldi loop builds an orthogonal basis of the right preconditioned quaternion Krylov subspace
\begin{equation}\label{e:right_KPAr0}
\Ks_m^R(\Aq\Pq^{-1},\rqq_0)=\texttt{span}\{\rqq_0,\Aq\Pq^{-1}\rqq_0,(\Aq\Pq^{-1})^2\rqq_0,\ldots,(\Aq\Pq^{-1})^{m-1}\rqq_0\},
\end{equation}
in which $\rqq_0$ is the residual $\rqq_0=\bq-\Aq\Pq^{-1}\uq_0$.
For the variable $\uq$ in \eqref{e:APu=b}, the right preconditioned QGMRES process minimizes the 2-norm of $\rqq_0=\bq-\Aq\Pq^{-1}\uq_0$, where $\uq_0$ belongs to
\begin{equation}\label{e:right_affine_u}
\uq_0+\Ks_m^R(\Aq\Pq^{-1},\rqq_0)=\uq_0+\texttt{span}\{\rqq_0,\Aq\Pq^{-1}\rqq_0,(\Aq\Pq^{-1})^2\rqq_0,\ldots,(\Aq\Pq^{-1})^{m-1}\rqq_0\},
\end{equation}
in which $\rqq_0$ is the residual $\rqq_0 =\bq-\Aq\Pq^{-1}\uq_0$.
Multiplying \eqref{e:right_affine_u} through to the left by $\Pq^{-1}$ yields the generic variable $\xq$ associated with a vector of the subspace \eqref{e:right_affine_u} belongs to the affine subspace
\begin{align*}
&\Pq^{-1}\uq_0+\Pq^{-1}\Ks_m^R(\Aq\Pq^{-1},\rqq_0)=\xq_0+\Pq^{-1}\Ks_m^R(\Aq\Pq^{-1},\rqq_0).
\end{align*}

We now have the following corollary by Lemma \ref{l:galerkin}.

\begin{cor}
Let $\Aq$ be an arbitrary square quaternion matrix and assume that $\Ls_m=\Aq\Ks_m$.
Then a vector $\tilde{\xq}$ is the result of an (oblique) projection method onto $\Ks_m^R$ orthogonally to $\Ls_m$ with the starting vector $\xq_0$ if and only if 
\begin{equation*}
R(\tilde{\xq})=\min_{\xq\in \xq_0+\Pq^{-1}\Ks_m^R}R(\xq),
\end{equation*}
where $R(\xq):= \|\bq-\Aq\Pq^{-1}\uq_0\|_2$.
\end{cor}

\begin{thm}\label{right_xm}
Let $\xq_m$ be the approximate solution obtained from the $m$th step of the right-preconditioned algorithm, and let $\rqq_m =\bq-\Aq\Pq^{-1}\uq_m=\bq-\Aq\xq_m$.
Then $\xq_m$ is of the form
$
\xq_m=\xq_0+\Pq^{-1}L_m(\Aq\Pq^{-1},\rqq_0)
$
and
$
\|\rqq_m\|_2=\|\rqq_0-\Pq^{-1} L_m(\Aq\Pq^{-1},\rqq_0)\|_2
  =\min_{j\leq m}\|\rqq_0-\Pq^{-1}\Aq L_j(\Aq\Pq^{-1},\rqq_0)\|_2.
$
\end{thm}

\begin{proof}
Recall the definition \eqref{e:LAv} that
 \begin{equation*}
 L_m(\Aq\Pq^{-1},\rqq_0)=\rqq_0\boldsymbol{\alpha}_0+\Aq\Pq^{-1}\rqq_0\boldsymbol{\alpha}_1+\cdots+(\Aq\Pq^{-1})^{m-1}\rqq_0\boldsymbol{\alpha}_{m-1}.
 \end{equation*}
The result  follows the fact that $\xq_m$ minimizes the 2-norm of the residual in the affine subspace $\xq_0+\Pq^{-1}\Ks_m(\Aq\Pq^{-1},\rqq_0)$ and the fact that $\Ks_m$
is the set of all vectors of the form $L_j(\Aq\Pq^{-1},\rqq_0)$ with  $j\leq m$.
\end{proof}

\begin{cor}
The approximate solution obtained by left or right-preconditioned QGMRES is of the form
$
\xq_m=\xq_0+L_m(\Pq^{-1}\Aq,\zq_0)=\xq_0+\Pq^{-1}L_m(\Aq\Pq^{-1},\rqq_0),
$
where $\zq_0=\Pq^{-1}\rqq_0$ and $L_m$ is the (right-hand side) linear combination of quaternion vectors.
$L_m$ minimizes the residual norm $\bq-\Aq\xq_m$ in the right preconditioning case,
and the preconditioned residual norm $\Pq^{-1}(\bq-\Aq\xq_m)$ in the left preconditioning case.
\end{cor}

\begin{proof}
Combining Theorem \ref{left_xm} and \eqref{e:left_affine_2}, we obtain directly for left-preconditioned QGMRES,
$
\xq_m=\xq_0+L_m(\Pq^{-1}\Aq,\zq_0)=\xq_0+\Pq^{-1}L_m(\Aq\Pq^{-1},\rqq_0),
$
where $\zq_0=\Pq^{-1}\rqq_0$.

On the other hand, for the quaternion affine subspace $\xq_0+\Pq^{-1}\Ks_m(\Aq\Pq^{-1},\rqq_0)$ of right-preconditioned QGMRES,
we deduce
\begin{align*}
&\xq_0+\Pq^{-1}\Ks_m(\Aq\Pq^{-1},\rqq_0)\\
=&\xq_0+\texttt{span}\{\Pq^{-1}\rqq_0,(\Pq^{-1}\Aq)\Pq^{-1}\rqq_0,(\Pq^{-1}\Aq)^2\Pq^{-1}\rqq_0,\ldots,(\Pq^{-1}\Aq)^{m-1}\Pq^{-1}\rqq_0\}.\\
=&\xq_0+\Ks_m(\Pq^{-1}\Aq, \Pq^{-1}\rqq_0).
\end{align*}
Then we have by Theorem \ref{right_xm},
$
\xq_m=\xq_0+L_m(\Pq^{-1}\Aq,\zq_0)=\xq_0+\Pq^{-1}L_m(\Aq\Pq^{-1},\rqq_0),
$
where $\zq_0=\Pq^{-1}\rqq_0$.
\end{proof}

The following theorem states a global convergence result for the right-preconditioned QGMRES.

\begin{thm}
Let $\Aq\Pq^{-1}$ be a diagonalizable quaternion matrix, $\Aq\Pq^{-1}=\Xq\boldsymbol{\Lambda}\Xq^{-1}$ and $\Xq^{- 1}\rqq_0:=[\xq_1,\xq_2,\ldots,\xq_n]^T$,
where $\boldsymbol{\Lambda}=
\diag(\boldsymbol{\lambda}_1,\boldsymbol{\lambda}_2,\ldots,\boldsymbol{\lambda}_n)$
is the diagonal matrix of eigenvalues.
Define
\begin{equation*}
  \epsilon^{(m)}=
  \min_{p\in \mathbb{P}_m^q,~p(0)=1}
  \max_{i=1,\ldots,n}
    |p(\hat{\boldsymbol{\lambda}}_i)|
    =\min_{q\in \mathbb{P}_m^q,~q(0)=1}
  \max_{i=1,\ldots,n}
    |q( \boldsymbol{\lambda}_i)|.
\end{equation*}
where
$
p(\hat{\boldsymbol{\lambda}}_i)= \boldsymbol{\alpha}_0+ \hat{\boldsymbol{\lambda}}_i\boldsymbol{\alpha}_1+ \cdots+\hat{\boldsymbol{\lambda}}_i^{m-1}\boldsymbol{\alpha}_{m-1},~~
\hat{\boldsymbol{\lambda}}_i:=\xq_i^{-1}\boldsymbol{\lambda}_i \xq_i,
$
and
$
q( \boldsymbol{\lambda} _i)
= \xq_i(\boldsymbol{\alpha}_0 + \xq_i^{-1}\boldsymbol{\lambda}_i \xq_i\boldsymbol{\alpha}_1+\cdots+\xq_i^{-1}\boldsymbol{\lambda}_i^{m-1}\xq_i\boldsymbol{\alpha}_{m-1})\xq_i^{-1}.
$
Then, the residual norm achieved by the $m$th step of QGMRES  satisfies the inequality
\begin{equation*}
  \|\rqq_m\|_2\leq\kappa_2(\Xq)\epsilon^{(m)}\|\rqq_0\|_2,
\end{equation*}
where $\kappa_2(\Xq):=\|\Xq\|_2\|\Xq^{-1}\|_2$.

\end{thm}
\begin{proof}
Let $L_{j+1}(\Aq\Pq^{-1},\rqq_0)$ satisfy the constraint $L_{j+1}(0,\rqq_0)=\rqq_0~(\boldsymbol{\alpha}_0=1)$,
and $\xq$ the vector from $\Ks_m^R$
to which it is associated via
$(\bq-\Axq)=L_{j+1}(\Pq^{-1}\Aq,\rqq_0)$.
Then
\begin{align*}
\|\bq-\Axq\|_2
&=\Xq\boldsymbol{D}_0\Xq^{-1}\rqq_0+\Xq\boldsymbol{\Lambda} \boldsymbol{D}_1\Xq^{-1} \rqq_0+\cdots+\Xq\boldsymbol{\Lambda}^j \boldsymbol{D}_j\Xq^{-1} \rqq_0 \\
&\leq\|\Xq\|_2\|\boldsymbol{D}_0+\boldsymbol{\Lambda} \boldsymbol{D}_1+\cdots+\boldsymbol{\Lambda}^j \boldsymbol{D}_j\|_2\|\Xq^{-1}\|_2\|\rqq_0\|_2.
\end{align*}
Since $\boldsymbol{\Lambda}$ and $\boldsymbol{D}_0, \boldsymbol{D}_1,\ldots,\boldsymbol{D}_j$ are diagonal,
$
\|\boldsymbol{D}_0+\boldsymbol{\Lambda} \boldsymbol{D}_1+\cdots+\boldsymbol{\Lambda}^j \boldsymbol{D}_j\|_2
\leq\max_{i=1,\ldots,n}|p(\hat{\boldsymbol{\lambda}}_i)|.
$
Since $\xq_m$ minimizes the residual norm over $\xq_0+\Ks_m^R$, then for any combining form $L_{j+1}(\Aq\Pq^{-1},\rqq_0)$,
\begin{align*}
&\| \bq-\Axq_m\|_2 \leq\|\bq-\Axq\|_2\\
\leq&\|\Xq\|_2\|\Xq^{-1}\|_2\|\rqq_0\|_2\min_{p\in \mathbb{P}_m^q,~p(0)=1}\max_{i=1,\ldots,n}
|p(\hat{\boldsymbol{\lambda}}_i)|\\
=&\|\Xq\|_2\|\Xq^{-1}\|_2\|\rqq_0\|_2\min_{q\in \mathbb{P}_m^q,~q(0)=1}
\max_{i=1,\ldots,n}|q( \boldsymbol{\lambda}_i)|.
\end{align*}
This yields the desired result,
$
\|\rqq_m\|_2\leq\kappa_2(\Xq)\epsilon^{(m)}\|\rqq_0\|_2.
$
\end{proof}

Notice that the sharp upper bound of residual norm for QGMRES and preconditioned QGMRES converges to $0$ by using quaternion polynomial and similar quaternion polynomial.
 Table \ref{Ex_QGMRES} shows the results of applying the Algorithms \ref{code:leftQGMRES} and \ref{code:rightQGMRES}
with SGS (SSOR with $\omega=1$) preconditioning \cite{jin2010, jin1992, jin2007} to compute a quaternion linear system $\Axq=\bq$.
Let $\Aq \in\mathbb{Q}^{500\times 500}$ be a random diagonally dominant matrix.
Each component of the right-hand side $\bq$ is a random vector.
The tolerance is set as $\delta=1.0e$-$6$.
Setting initial solution is zero.
From Table \ref{Ex_QGMRES},  
we can observe that \texttt{QGMRES$_{lp}$} and \texttt{QGMRES$_{rp}$} converge in fewer iterations and cost less CPU time than \texttt{GMRES} and \texttt{QGMRES}, while their residual errors are comparable with each other.

\begin{table}[!h]
\tabcolsep 0pt \caption{}\label{Ex_QGMRES}  \vspace*{-25pt}
\begin{center}
\def\temptablewidth{1\textwidth}
{\rule{\temptablewidth}{1pt}}
\begin{tabular*}{\temptablewidth}{@{\extracolsep{\fill}}lcccc}
 \multirow{1}*{$Algorithm$} &  Iter &  CPU time & Residual  \\ \hline
  $\texttt{GMRES}$ &14& 0.2955 &7.2450e-07  \\
 $\texttt{QGMRES}$ &14& 0.1834 &3.2147e-07\\
 $\texttt{QGMRES$_{lp}$}$  &3& 0.0737 &1.4067e-07 \\
 $\texttt{QGMRES$_{rp}$}$  &3& 0.0813&  1.6287e-07  \\
\end{tabular*}
{\rule{\temptablewidth}{1.2pt}}
\end{center}
\end{table}

\subsection{Flexible QGMRES}\label{s:FQGMRESalgorithm}
We now examine the right preconditioned QGMRES algorithem of the flexible version.
According to the Algorithm \ref{code:rightQGMRES}, the approximate solution $\xq_m$ is expressed as
$
\xq_m=\xq_0+\Pq^{-1}\textbf{V}_m\yq_m=\xq_0+\sum^m_{j=1}\Pq^{-1}\vq_j\yq_j,
$
i.e., $\xq_m-\xq_0$ is a linear combination of the preconditioned vectors
$\Pq^ {-1}\vq_1,\Pq^{-1}\vq_2,\cdots,\Pq^{-1}\vq_m,$
which can be obtained from line 3 in Algorithm \ref{code:rightQGMRES}.
Since they are all obtained by applying the same preconditioning matrix $\Pq^{-1}$ to the $\vq_i'$s, it is not necessary to save them.
Thus, we only need to apply $\Pq^{-1}$ to the linear combination of the $\vq_i'$s, i.e., to $\textbf{V}_m\yq_m$ in line 13 of Algorithm \ref{code:rightQGMRES}.
Now we assume that the preconditioner can change at each step.
The flexible quaternion GMRES algorithm is described as follows (Algorithm \ref{code:FQGMRES}).

\begin{algorithm}
\setcounter{algorithm}{3}
\caption{Flexible QGMRES (FQGMRES)}
\label{code:FQGMRES}
\begin{algorithmic}[1]

\STATE Compute $\rqq_0=\bq-\Axq_0$, $\beta:=\|\rqq_0\|_2\ne 0$, $\vq_1:=\rqq_0/\beta$, and $\Pq_{j}:=I$
\FOR{ $j=1,2,\ldots,m$}\label{fgmresq2}
\STATE Compute  $\zq_{j}:=\Pq_{j}^{-1}\vq_{j}$
\STATE Compute $\boldsymbol{\omega}_j:=\Aq\zq_j$ \label{fgmresq4}
\FOR{ $i=1,2,\ldots,j$}
\STATE  $\hq_{ij}=\langle\boldsymbol{\omega}_j,\vq_i\rangle$;
\STATE  $\boldsymbol{\omega}_j:=\boldsymbol{\omega}_j-\hq_{ij}\vq_i$
\ENDFOR
\STATE $\hq_{j+1,j}=\|\boldsymbol{\omega}_j\|_2$
\STATE {If $\hq_{j+1,j}=0$ set $m:=j$ and go to line \ref{fgmresq12}}
\STATE $\vq_{j+1}=\boldsymbol{\omega}_j/\hq_{j+1,j}$
\ENDFOR\label{fgmresq10}
\STATE Update the flexible preconditioner $\Pq_{j}$
\STATE Define $(m+1)\times m$ quaternion Hessenberg matrix $\bar{\Hq}_m=[\hq_{ij}]_{1\leq i\leq m+1, 1\leq j\leq m}$.\\[-0.5pt]\label{fgmresq12}
\STATE Compute $\yq_m$ the minimizer of $\|\beta \textbf{e}_1-\bar{\Hq}_m\yq\|_2$
and $\xq_m=\xq_0+\Zq_m\yq_m$.\label{fgmresq13}
\end{algorithmic}
\end{algorithm}

The $m$th iteration of the FQGMRES algorithm updates a decomposition of the form
\begin{eqnarray}
\Aq\Zq_m &=& \Vq_m\Hq_m+\boldsymbol{\omega}_m\eq_m^*\label{e:AZ=VH_residual}\\
         &=& \Vq_{m+1}\bar{\Hq}_m,\label{e:AZ=VH}
\end{eqnarray}
where $\Zq_{m}\in\mathbb{Q}^{n\times m}$, $\Vq_{m+1}\in \mathbb{Q}^{n\times (m+1)}$, and $\bar{\Hq}_{m} \in\mathbb{Q}^{(m+1)\times m}$.
More specifically, in the above decomposition \eqref{e:AZ=VH},
$\bar{\Hq}_m$ denote the $(m+1)\times m$ quaternion Hessenberg matrix whose nonzero entries $\hq_{ij}$, and $\Hq_m$ be the quaternion matrix obtained from $\bar{\Hq}_m$ by deleting its last row;
$\Vq_{m+1}:=[\vq_1,\vq_2,\ldots,\vq_{m+1}]$ has orthonormal columns $\vq_{i}$, $i = 1,2,\ldots,m+1$;
$\eq_m$ is the $m$th column of the identity matrix;
$\Zq_{m}:=[\zq_1,\zq_2,\ldots,\zq_m]$ has columns $\zq_{i} =\Pq_{i}^{-1}\vq_{i}$, $i = 1,2,\ldots,m$.
Since the columns of $\Zq_{m}$ already include the contribution of the variable preconditioners $\Pq^{-1}_{i}\Aq$ for $i = 1,2,\ldots, m$,
they form a basis for the vector $ {\xq}_{m}$.
Therefore, at the $m$th step of FQGMRES, ${\xq}_{m}$ is approximated by the following vector
\begin{equation}\label{e:fgmresq_xm}
\xq_m=\xq_0+\sum^m_{i=1}(\Pq_{i}^{-1}\vq_{i})\yq_i=\xq_0+\Zq_{m}\yq_{m},
\end{equation}

An optimality property similar to the one which defines QGMRES can be proved.
Consider the residual vector for any quaternion vector $\zq = \xq_0 + \Zq_m \yq$ in the affine space $\xq_0+\texttt{span}\{\Zq_m\}$.
Due to decomposition \eqref{e:AZ=VH} and the properties of the matrices appearing therein, we have the relations
\begin{align}
\bq-\Aq\zq&=\bq-\Aq(\xq_0+\Zq_m\yq)\nonumber ~=~\rq_0-\Aq\zq_m\yq\nonumber \\
\label{e:fgmres_residual}& =\beta \vq_1-\textbf{V}_{m+1}\bar{\Hq}_m\yq  ~=~\textbf{V}_{m+1}(\beta\textbf{e}_1-\bar{\Hq}_m\yq),
\end{align}
where $\beta=\|\textbf{r}_0\|_2\ne 0,~\vq_1=\textbf{r}_0/\beta$.
Define
\begin{equation}\label{e:Jy}
J_m(\yq):=\|\bq-\textbf{Az}\|_2=\|\bq-\Aq(\xq_0+\textbf{Z}_m\yq)\|_2,
\end{equation}
one can observe from the orthogonality of the columns of $\textbf{V}_{m+1}$ that
\begin{equation}\label{e:Jy2}
  J_m(\yq)=\|\beta \textbf{e}_1-\bar{\Hq}_m\yq\|_2.
\end{equation}
Thus, the approximate solution ${\xq}_m=\xq_0+\Zq_{m}\yq_{m}$ obtained at the $m$th iteration of FQGMRES minimizes the residual norm over all the vectors in
\begin{equation}\label{RZm}
 \Ks_m(\Zq_{m})=\texttt{span}\{\Pq_{1}^{-1}\vq_{1},\ldots,\Pq_{m}^{-1}\vq_{m}\},
\end{equation}
where
$$
\yq_{m}=\arg\min_{\yq\in \mathbb{Q}^{m}}\|\beta\textbf{e}_1-\bar{\Hq}_m\yq\|_{2}.
$$
The approximation subspace $\Ks_m(\Zq_{m})$ can be regarded as a preconditioned quaternion Krylov subspace,
where the preconditioner is implicitly defined by successive applications of the matrices $\Pq_{i}^{-1}\Aq$ to the linearly independent vectors $\vq_{i}$.
The following result is proved.

\begin{thm}
The approximate solution $\xq_{m}$ obtained at step $m$ of FQGMRES minimizes the residual norm $\|\bq-\Aq\xq_m\|$ over $\xq_0+\texttt{span}\{\Zq_{m}\}$.
\end{thm}

For the case of a breakdown in FQGMRES, the only possibility is when $\hq_{j+1,j}=0$ at a given step $j$ in the quaternion Arnoldi loop.
Then the algorithm stops because the next quaternion Arnoldi vector cannot be generated.

\begin{thm}
Assume that $\beta=\|\rqq_0\|\neq 0$ and that $j-1$ steps of FQGMRES have been successfully performed,
i.e., that $\hq_{i+1,i}\neq0$ for $i<j$.
In addition, assume that the matrix $\Hq_j$ is nonsingular.
Then the approximate solution $\xq_j$ is exact, if and only if $\hq_{j+1,j}=0$.
\end{thm}

\begin{proof}
If $\hq_{j+1,j}=0$, then $\Aq\Zq_j=\Vq_j\Hq_j$ and
\begin{equation}\label{e:breakdown_1}
J_j(\yq)=\|\beta \vq_1-\Aq\Zq_j\yq_j\|_2=\|\beta \vq_1-\Vq_j\Hq_j\yq_j\|_2=\|\beta \eq_1-\Hq_j\yq_j\|_2.
\end{equation}
Since $\Hq_j$ is a nonsingular quaternion matrix, \eqref{e:breakdown_1} implies
$
\yq_j=\Hq_j^{-1}(\beta \eq_1)
$
such that the corresponding minimum norm reached $0$.
Thus, $\xq_j$ is exact.

Conversely, if $\xq_j$ is exact, it follows from \eqref{e:AZ=VH_residual} and \eqref{e:fgmres_residual} that
\begin{equation}\label{e:breakdown_2}
 0=\bq-\Aq\xq_j=\Vq_j(\beta \eq_1-\Hq_j\yq_j)+\vq_{j+1}\eq_j^*\yq_j.
\end{equation}
By contradiction, we assume that $\vq_{j+1}\neq 0$.
Since $\vq_1,\vq_2,\ldots,\vq_{m+1}$ form an orthogonal system and $\Hq_j$ is a nonsingular quaternion Hessenberg matrix,
then by \eqref{e:breakdown_2} we have all components of $\yq_j=0$, i.e., $\beta=0$,
which contradicts $\beta\neq0$.
Thus, $\hq_{j+1,j}=0$.
\end{proof}

\section{Quaternion Total Variation Regularization Model}\label{s:qctv}

In this section, we introduce a novel quaternion regularization model to solve ill-posed inverse problems \eqref{ill} from color image processing within a variational framework.
Firstly, we give the following definition.
Let $\Xq$ be an $n \times n$ quaternion matrix which produced by a color image \cite{jnw19} and the vector $\xq\in\mathbb{Q}^{n^2}$ obtained by stacking the columns of $\Xq$.
For convenience, we collectively denote $n^2$ as $N$.
Suppose $\Aq$ is a blur operator and $\eq$ denotes unknown noise.
Then $\bq$ is an observed color image.
Let $D_h$ and $D_v$ denote the forward finite
difference approximations of the horizontal and vertical first derivative operators, respectively.
The subscript $[\cdot]_{i}$ denotes the $i$th element of a vector.
The total variation of color image $\xq$ is defined by
\begin{equation*}
{\rm QTV}(\xq)=\sum_{i=1}^{N} \big( |D_{h}\xq_{i}|^{2}+|D_{v}\xq_{i}|^{2} \big)^{\frac{1}{2}}.
\end{equation*}
In fact, TV$(\xq)$ is the isotropic total variation of  $\xq$, which measures the magnitude of the discrete gradient of $\xq$ in the $\ell^{1}$ norm.
Then the solution of \eqref{ill} can be computed by solving the following quaternion total variation regularization model
\begin{equation}\label{e:TVQ}
  \xq_{\text{QTV},\lambda} = \arg \min _{\xq\in \mathbb{Q}^{N}}
\|\Aq\xq-\bq\|_{2}^{2} + \lambda\text{QTV}(\xq),
\end{equation}
which is a color image restoration model.

Since the convex optimization problem \eqref{e:TVQ} is very challenging to solve,
the QIRN strategy was proposed to solving this problem.
This idea was first proposed for total variation regularization in \cite{Wohlberg2007},
which consist of the solution of a sequence of penalized weighted least-squares problems with diagonal weighting matrices incorporated into the regularization term and dependent on the previous approximate solution.
In addition, we proposed the quaternion collaborative total variation regularization relative to quaternion vectors in the process of demonstration to represent the regularization term in model \eqref{e:TVQ}.

\subsection{Quaternion collaborative total variation regularization}\label{s:qctvregularization}

Define a mapping
\begin{equation}\label{e:mapping}
\begin{split}
 \Psi: \mathbb{Q}^{N}&\longrightarrow  \mathbb{R}^{N\times 4} \\
 \xq &\longmapsto \Psi(\xq),\\
\end{split}
\end{equation}
where the column vectors of $\Psi(\xq)$ are the real part and three imaginary parts of $\xq$, denoted as $x^{0},x^{r},x^{g},x^{b}$.
For the represented color image, $x^{r},x^{g},x^{b}$ are the three color channels of red, green, and blue.
Consider the gradient of this color image as a three-dimensional matrix or tensor with dimensions corresponding to the spatial extent, the intensity differences between neighboring pixels, and the spectral channels.
For discrete isotropic total variation regularization term $\text{TV}(\xq)$,
we can define a collaborative total variation of $\Psi(\xq)$ by \eqref{l_pqr} equivalent to $\text{TV}(\xq)$.

\begin{thm}\label{thm:ctv}
The total variation regularization for a quaternion vector $\xq\in\mathbb{Q}^{N}$ is equivalent to the collaborative total variation regularization for its map $\Psi(\xq)\in \mathbb{R}^{N\times 4}$ defined as a $\ell^{2,2,1}(col,der,pix)$ norm
\begin{equation}\label{e:CTV_Q}
\|D\Psi(\xq)\|_{2,2,1}
=\sum_{i=1,\ldots,N}\Big(\sum_{j=h,v}\big(\sum_{k=0,r,g,b}|D_{j}x_{i}^{k}|^{2}\big)\Big)^{\frac{1}{2}},
\end{equation}
where the $D_h$ and $D_v$ denote the forward finite
difference approximations of the horizontal and vertical first derivative operators, respectively; $x^{k}$ are the color channels of $\xq$;
the subscript $[\cdot]_{i}$ denotes the $i$th element of this vector.
\end{thm}

\begin{proof}
By the definition of collaborative norm, we have
\begin{align}
 \nonumber &\|D\Psi(\xq)\|_{2,2,1} (col,der,pix)\\\nonumber
=&\sum_{i=1,\ldots,N} \Big(\sum_{j=h,v}  \big(|D_{j}x_{i}^{0}|^{2}+|D_{j}x_{i}^{r}|^{2}+|D_{j}x_{i}^{g}|^{2}+|D_{j}x_{i}^{b}|^{2}\big)^\frac{2}{2}  \Big)^{\frac{1}{2}}\\\nonumber
=&\sum_{i=1,\ldots,N}   \Big(  \big(|D_{h}x^{0}_{i}| ^{2}+ |D_{h}x^{r}_{i}|^{2}+ |D_{h}x^{g}_{i}|^{2}+ |D_{h}x^{b}_{i}|^{2} \big) \\\nonumber
 &~~~~+                   \big(|D_{v}x^{0}_{i}|^{2}+ |D_{v}x^{r}_{i}|^{2}+ |D_{v} x^{g}_{i}|^{2}+ |D_{v} x^{b}_{i}|^{2}\big) \Big)^{\frac{1}{2}}   \\\nonumber
=&\sum_{i=1,\ldots,N} \Big( \sum_{k=0,r,g,b}{|D_{h}x^{k}_{i}| ^{2}}+ \sum_{k=0,r,g,b}{|D_{v}x^{k}_{i}| ^{2}} \Big)^{\frac{1}{2}}   \\\label{e:ctvterm1}
=&\sum_{i=1,\ldots,N} \big( |D_{h}\xq_{i}|^{2}+|D_{v}\xq_{i}|^{2} \big)^{\frac{1}{2}}
= \text{TV}(\xq).
\end{align}
Thus, the proof is completed.
\end{proof}

In fact, if we define the collaborative total variation regularization for $\Psi(\xq)$ as $\ell^{2,2,1}(der,col,pix)$ norm,
i.e., exchange the order of two dimensions of color channels and derivatives,
the theorem still holds.
In this paper, we only choose the first case, i.e., the order is $col, der, pix$.
For convenience, we rewrite $\|D\Psi(\xq)\|_{2,2,1} (col,der,pix)$ in the short form $\|D\Psi(\xq)\|_{2,2,1}$ in the following discussion.
In addition, the real part is a zero matrix if the practice color image is processed, i.e., the color channel here has only three channels $x^{r}, x^{g}, x^{b}$.

\subsection{Quaternion iteratively reweighted norm method}\label{s:qarnoldi}
We now focus on approximating the quaternion TV regularization operator in \eqref{e:TVQ}.
Consider the matrix
\begin{equation}\label{e:d1d}
  D_{1d}=\left[
      \begin{array}{ccccc}
         1 & -1 &   &  \\
          & \ddots & \ddots &  \\
          &   & 1 & -1 \\
      \end{array}
    \right]\in \mathbb{R}^{(n-1)\times n},
\end{equation}
which is a scaled finite difference approximation for the one-dimensional derivative \cite{Hansen2010, Hansen2007}.
Then the discrete first derivatives in the horizontal and vertical directions
are given by 
$$D_{h}\xq=(D_{1d}\otimes I)\xq\in \mathbb{Q}^{\tilde{N}},\quad D_{v}\xq=(I\otimes D_{1d})\xq\in \mathbb{Q}^{\tilde{N}},$$
for $\tilde{N}=(n-1)\times n$, where $I$ is the identity matrix of size $n$.
For the given $\xq$, one takes
\begin{equation}\label{e:Dhv}
  D_{hv}=\left[
           \begin{array}{c}
             D_{h} \\
             D_{v} \\
           \end{array}
         \right]\in \mathbb{R}^{2\tilde{N}\times N},
\end{equation}
then we have
\begin{equation*}
  \|D_{hv}\xq\|_{2}^{2}= \sum_{i=1,\ldots,\tilde{N}}|[D_{h}\xq]_{i}|^{2}+|[D_{v}\xq]_{i}|^{2}.
\end{equation*}

Now we consider the diagonal weighting matrix
\begin{equation}\label{w2d}
W_{2d}=\texttt{diag} \big( \tilde{W}_{2d},\tilde{W}_{2d} \big)\in \mathbb{R}^{2\tilde{N} \times 2\tilde{N}},
\end{equation}
where
\begin{equation}\label{tildew2d}
\large \tilde{W}_{2d}
:=\tilde{W}_{2d}(D_{hv}\xq)
=\texttt{diag}\bigg( \Big(\big(~|[D_{h}\xq]_{i}|^{2}+|[D_{v}\xq]_{i}|^{2}\big)^{-\frac{1}{4}} \Big)_{i=1,2,\ldots,\tilde{N}} \bigg)
\in \mathbb{R}^{\tilde{N} \times \tilde{N}}.
\end{equation}
By the Theorem \ref{thm:ctv}, we can verify that the optimal regularization matrix to choose in order to recover the TV regularization operator \cite{Calvetti2002} is
$P=W_{2d}D_{hv}\in \mathbb{R}^{2\tilde{N} \times \tilde{N}}$ since
\begin{align*}
& \large \|P\xq\|_{2}^{2}
  =\left\| W_{2d}D_{hv}\xq \right\|_{2}^{2}\\[0.15in]
=&\left\|\left[
             \begin{array}{cc}
               \tilde{W}_{2d} & 0 \\
                0 & \tilde{W}_{2d} \\
             \end{array}
           \right]
           \left[
         \begin{array}{c}
           D_{h} \\
           D_{v}\\
         \end{array}
       \right]
       \xq\right\|_{2}^{2}
  =\left\|\left[
        \begin{array}{c}
          \tilde{W}_{2d}D_{h}\xq\\
         \tilde{W}_{2d}D_{v}\xq\\
        \end{array}
      \right]
   \right\|_{2}^{2}\\[0.15in]
= &\sum_{i=1,\ldots,\tilde{N}}\big( {[ \tilde{W}_{2d}]}_{ii} ^{2} (|[D_{h}\xq]_{i}|^{2}+|[D_{v}\xq]_{i}|^{2}) \big)\\
=&\sum_{i=1,\ldots,\tilde{N}} \big(~|[D_{h}\xq]_{i}|^{2}+|[D_{v}\xq]_{i}|^{2}\big)^{-\frac{1}{2}}
   \big(~|[D_{h}\xq]_{i}|^{2}+|[D_{v}\xq]_{i}|^{2}\big)  \\
=&\sum_{i=1,\ldots,\tilde{N}}\big(~|[D_{h}\xq]_{i}|^{2}+|[D_{v}\xq]_{i}|^{2}\big)^{\frac{1}{2}}
   = \|D\Psi(\xq)\|_{2,2,1}  =\text{TV}(\xq).
\end{align*}
Thus, for the solution of quaternion TV problem, we can solve a sequence of regularized problems with rescaled penalization terms expressed as reweighted $\ell_2$ norms,
whose weights are iteratively updated using a previous approximation of the quaternion solution, i.e.,
\begin{equation}\label{e:appro_TV}
  \xq_{P,\lambda}=\arg\min_{\xq\in\mathbb{Q}^{N}}\|\Aq\xq-\bq\|_{2}^{2}+\lambda \|P \xq\|_{2}^{2},
\end{equation}
where $P=W_{2d}D_{hv} $.
Clearly, $\xq_{\text{TV},\lambda}= \xq_{P,\lambda}$.

Define $P\xq_{\text{QTV},\lambda}:=\hat{\yq}$.
The solution of \eqref{e:appro_TV} can be equivalently expressed as a standard form:
\begin{equation}\label{bar_y1}
\hat{\yq}=\arg\min_{\hat{\yq}\in\mathbb{Q}^{N}}\|\hat{\Aq}\hat{\yq}
-\hat{\bq}\|_{2}^{2}+\lambda \|\hat{\yq}\|_{2}^{2},
\end{equation}
where
\begin{equation}\label{bar_y}
\hat{\Aq}=\Aq P^{\dagger}, \quad  \hat{\bq}=\bq-\Aq\xq_0.
\end{equation}
Here $P^{\dagger}$ is the Moore–Penrose pseudoinverse of $P$.
Notice this regularization matrix $P$ is a real matrix.
By the definition of the TV regularization operator $P=W_{2d}D_{hv}$ and \eqref{tildew2d}, it is relative to the updated solution.
Thus, the process of solving $\hat{\yq}$ is actually flexible QGMRES.
We have the following algorithm, which leads to the decomposition
\begin{equation*}
\hat{\Aq}\Zq_m = \Zq_{m+1}\bar{\Hq}_m,
\end{equation*}
where $\Zq_{m+1}:=[\zq_1,\zq_2,\ldots,\zq_{m+1}]\in \mathbb{Q}^{n\times (m+1)}$ has orthonormal columns that span the quaternion Krylov subspace $\Ks_m(\hat{\Aq},\zq_0)$.

\begin{algorithm}[H]
\setcounter{algorithm}{2}
\caption{QTV-FQGMRES}
\label{code:ctv_FQGMRESinverse}
\begin{algorithmic}[1]
\STATE Input $\hat{\Aq}$, $\hat{\bq}$ as in \eqref{bar_y} and take $W_{(1)}=I$, $\vq_1=\bq/\|\bq\|_2$.
\FOR{ $j=1,2,\ldots,m$}\label{ctvfgmresinv3}
\STATE Compute  $\zq_{j}:=P_{j}^{\dagger}\vq_{j}=(W_{(j)}D)^{\dagger}\vq_{j}$
\STATE Compute $\boldsymbol{\omega}_j:= \hat{\Aq} \zq_j$\label{ctvfgmresinv5}
\FOR{ $i=1,2,\ldots,j$}
\STATE  $\hq_{ij}=\langle\boldsymbol{\omega}_j,\vq_i\rangle$;
\STATE  $\boldsymbol{\omega}_j:=\boldsymbol{\omega}_j-\hq_{ij}\vq_i$
\ENDFOR
\STATE $\hq_{j+1,j}=\|\boldsymbol{\omega}_j\|_2$
\STATE {If $\hq_{j+1,j}=0$ set $m:=j$ and go to line \ref{ctvfgmresinv15}}
\STATE $\vq_{j+1}=\boldsymbol{\omega}_j/\hq_{j+1,j}$
\STATE Update the weights $W_{(j+1)}$
\ENDFOR\label{ctvfgmresinv14}
\STATE Define $(m+1)\times m$ quaternion Hessenberg matrix $\bar{\Hq}_m=[\hq_{ij}]_{1\leq i\leq m+1, 1\leq j\leq m}$.\\[-0.5pt]\label{ctvfgmresinv15}
\STATE Compute $\yq_m$ the minimizer of $\|\beta \textbf{e}_1-\bar{\Hq}_m\yq\|_2$
and $ \hat{\yq}_m=\yq_0+\textbf{Z}_m \yq_m$.\label{ctvfgmresinv16}
\end{algorithmic}
\end{algorithm}

The main cost of Algorithm \ref{code:ctv_FQGMRES} is computing the pseudoinverse of preconditioner $P_j$.
If $\tilde{N}$ is large, such cost is unacceptable.
Now, we consider search a approximation of problem \eqref{e:appro_TV} belonging to $\xq_0 + \Ks_m(\Aq,\rqq_0)$ by generalized Arnoldi Tikhonov method \cite{Berisha2014, Calvetti2000}.

For the quaternion Arnoldi process in $\xq_0 + \Ks_m(\Aq,\rqq_0)$, we obtain the decomposition $\AVq_m = \Vq_{m+1}\bar{\Hq}_m$.
Default $\xq_0={\bf 0}$ for simplicity.
Then replacing $\xq=\Vq_m\yq$ into \eqref{e:appro_TV} yields the reduced minimization problem
\begin{equation}\label{e:reducedproblem}
  \yq_m=\arg\min_{\yq\in\mathbb{Q}^{N}}\| \bar{\Hq}_m\yq-\Vq_{m+1}^*\rqq_0\|_{2}^{2}+\lambda \|P\Vq_m \yq\|_{2}^{2}.
\end{equation}
At each step of the Arnoldi algorithm, instead of solving \eqref{e:reducedproblem} directly, we consider
the following equivalent reduced-dimension least squares formulation
\begin{equation}\label{e:reduced2}
  \yq_m=\arg\min_{\yq\in\mathbb{Q}^{N}}\left\|\left[
         \begin{array}{c}
           \bar{\Hq}_m \\
          \sqrt{\lambda}P\Vq_m\\
         \end{array}
       \right]\yq-
       \left[
         \begin{array}{c}
           c \\
           0\\
         \end{array}
       \right]
       \right\|_{2}^{2},
\end{equation}
where $c=\|\rqq_0\|_2$ $e_1$.
For the typically tall coefficient matrix in \eqref{e:reduced2}, one can compute a QR factorization $P\Vq_m=\Qq_m\Rq_m$, where $\Qq_m$ has orthonormal columns and $\Rq$ is an upper triangular matrix.
Then replacing the matrix $P\Vq_m$ in \eqref{e:reduced2} by $\Rq_m$.
Thus, we have the following algorithm.

\begin{algorithm}[H]
\setcounter{algorithm}{3}
\caption{Improved QTV-FQGMRES}
\label{code:ctv_FQGMRES}
\begin{algorithmic}[1]
\STATE Input $\Aq$, $\bq$ and take $W_{(1)}=I$, $\vq_1=\bq/\|\bq\|_2$.
\FOR{ $m=1,2,\ldots,N$  until  $\|\bq-\Aq\xq_m\|\le tol$}\label{ctvfgmres3}
\STATE Compute the quaternion Hessenberg matrix $\bar{\Hq}_m$ and $\Vq_m$ in \eqref{e:AV=VH} by quaternion Arnoldi method.
\FOR{ $j=1,2,\ldots,m$}
\STATE Compute the least squares solution $\yq_m$ in \eqref{e:reduced2} by Algorithm \ref{code:ctv_FQGMRESinverse}, $P_{j}=W_{(j)}D$.
\ENDFOR
\ENDFOR\label{ctvfgmres14}
\end{algorithmic}
\end{algorithm}

\section{Numerical Experiments}\label{s:Experiments}

This section compares the proposed Algorithms with
the classical Algorithms for quaternion linear systems.
All experiments were implemented by MATLAB (R2016a) on a personal computer with 2.30 GHz Intel Core i7-10510U processor and 12 GB memory.
All algorithms to be compared with each other are listed as follows.\\

\begin{itemize}
\item \texttt{GMRES}-- The classical GMRES for solving linear systems.
\item \texttt{QGMRES}--Structure-preserving quaternion GMRES method in Algorithm \ref{code:QGMRES}.
\item \texttt{IRhybrid-FGMRES(IRFGMRES)}--Hybrid version of FGMRES algorithm for square systems in \cite{Gazzola2015, Gazzola2019}.
\item \texttt{FQGMRES}-- Flexible quaternion GMRES method in Algorithm \ref{code:FQGMRES}.
\item \texttt{QTV-FQGMRES}--Flexible quaternion GMRES method for quaternion total variation regularization model in Algorithm \ref{code:ctv_FQGMRES}.
\item \texttt{SV-TV}--Color image restoration model in \cite{jnw19}.\\
\end{itemize}

\begin{example}
In the first example, we compare the performance of the proposed algorithm with the traditional methods for sparse quaternion matrices.
Let $\Aq \in\mathbb{Q}^{N\times N}$ be of the form with its four parts  $A_0$, $A_1$, $A_2$ and $A_3$  being  the order-$3000$ principle submatrices of  $\texttt{bcspwr10}$, $\texttt{af23560}$, $\texttt{rw5151}$ and $\texttt{rdb5000}$ {\rm(}from Matrix Market\footnote{https://math.nist.gov/MatrixMarket}{\rm)}.
Each component of the right-hand side $\bq$ is a vector of all ones.
The tolerance is set as $\delta=1.0e$-$6$.
\end{example}

The above algorithms are applied to solve the large quaternion system  $\Aq\xq=\bq$.
The numerical results are listed in Table \ref{table1}.
``Iter'' denotes either the times of restarting for restarted algorithms or the times of iteration for algorithms without restarting.
n.c. denote the algorithms not converge within the maximum number of iterations. In the case of similar residuals, {\nobreak IRFGMRES} and  {\nobreak FQGMRES} work better. Algorithms {\nobreak GMRES} and {\nobreak QGMRES } cost more cpu time.
For the algorithm FQGMRES, define preconditioner $\Pq_{j}=\diag(\sqrt{|\bq-\Aq \xq_j|})$
in each iteration $j$.
We observed that the Algorithm \texttt{IRFGMRES} and \texttt{FQGMRES} converge in fewer iterations and less CPU time.

\begin{table}[!h]
\tabcolsep 0pt
\caption{ Numerical results of GMRES, QGMRES, IRFGMRES and FQGMRES.}\label{table1} \vspace*{-10pt}
\begin{center}
\def\temptablewidth{1\textwidth}
{\rule{\temptablewidth}{1pt}}
\begin{tabular*}{\temptablewidth}{@{\extracolsep{\fill}}lcccc}
  {$Algorithm$}&
Iter &  CPU time& \ \ Residual \\ 
  \hline
                    $\texttt{ GMRES }$ & 254 & 1.7085e+02& 9.5702e-07 \\
                    $\texttt{ QGMRES }$ &  233 & 1.8332e+02& 9.6491e-07 \\
                     $\texttt{ IRFGMRES }$ & 249 &\textbf{5.5092}& 1.0052e-06 \\
                     $\texttt{ FQGMRES}$ &  \textbf{226} & 8.9000& \textbf{9.4253e-07} \\
\end{tabular*}
{\rule{\temptablewidth}{0.5pt}}
\end{center}
\end{table}

\begin{example}
In this example, we solve the quaternion system for the signal matrix. A three-dimensional signal can be denoted by a quaternion function of time,
$\xq(t)=x_r(t)\iq+x_g(t)\jq+x_b(t)\kq$, where  $x_r(t), x_g(t)$, and $x_b(t)$ are real functions
(for example, they refer to
 the  red, green, and blue  channels, respectively).
We are interested in determining  quaternion filters
 $\{\wq{(s)}\}_{s=0}^n$,  where  $\wq{(s)}=w{(s)}_0+w{(s)}_r\iq+w{(s)}_g\jq+w{(s)}_b\kq$
on the input signal $\xq(t)=x_r(t)\iq+x_g(t)\jq+x_b(t)\kq$
such that the filtered output can match with the target signal
$\yq(t)=y_r(t)\iq+y_g(t)\jq+y_b(t)\kq$. More precisely, we have
\begin{equation}\label{e:ywx}
\yq(t)=\sum\limits_{s=0}^{n} \xq(t-s)*\wq{(s)} .
\end{equation}

The numerical results of applying {\nobreak GMRES}, {\nobreak QGMRES}, {\nobreak IRFGMRES} and {\nobreak QTV-FQGMRES} to solve the quaternion system for the signal matrix are listed in Table $\ref{table:signal}$, in which the notation has the following meanings:
``Iter''  denotes the average iteration, ``CPU time''  and  ``Residual'' denotes the CPU time and relative residual error required by {\nobreak GMRES}, {\nobreak QGMRES}, {\nobreak IRFGMRES} and {\nobreak QTV-FQGMRES}, respectively. We compare the performance of the proposed algorithm with the above methods for quaternion signal matrix.
We can see from the Table $\ref{table:signal}$ that
QTV-FQGMRES converges in fewer iterations and costs less CPU time than other algorithms, while their  residual errors are comparable with each other.
We can see that quaternion algorithms cost about $1/3$ iteation steps and $4/5$ CPU time of coresponding real ones.
Two quternion algorithms QTV-FQGMRES and QGMRES cost almost the same iteration steps, but the former is about 12 times faster than the later.

Especially, the convergence curves of IRFGMRES and QTV-FQGMRES are shown in Figure $\ref{fig:signal_iter}$, which indicates that QTV-FQGMRES stops earlier than IRFGMRES. QTV-FQGMRES stops the iteration at step 104 and IRFGMRES stops the iteration at step 297, that is, 193 iteration steps are saved.

\begin{table}[!h]\label{table:signal}
\tabcolsep 0pt
\caption{ Numerical results of GMRES, QGMRES, IRFGMRES and QTV-FQGMRES.}\label{table2} \vspace*{-10pt}
\begin{center}
\def\temptablewidth{1\textwidth}
{\rule{\temptablewidth}{1pt}}
\begin{tabular*}{\temptablewidth}{@{\extracolsep{\fill}}lcccc}
 {$Algorithm$}&
Iter &  CPU time& \ \ Residual \\ 
  \hline
                    $\texttt{ GMRES }$ &  300 &  16.9492& 5.3672e-07 \\
                    $\texttt{ QGMRES }$ &  105 & 11.3170& \textbf{2.2624e-10} \\%
                     $\texttt{ IRFGMRES }$ & 297 &1.1920& 8.1931e-07 \\
                     $\texttt{ QTV-FQGMRES}$ &  \textbf{104} &  \textbf{0.9938} & 4.9865e-07\\
\end{tabular*}
{\rule{\temptablewidth}{0.5pt}}
\end{center}
\end{table}

 \begin{figure}[!h]\label{fig:signal_iter}
\begin{center}
\includegraphics[height=7.5cm, width=12cm]{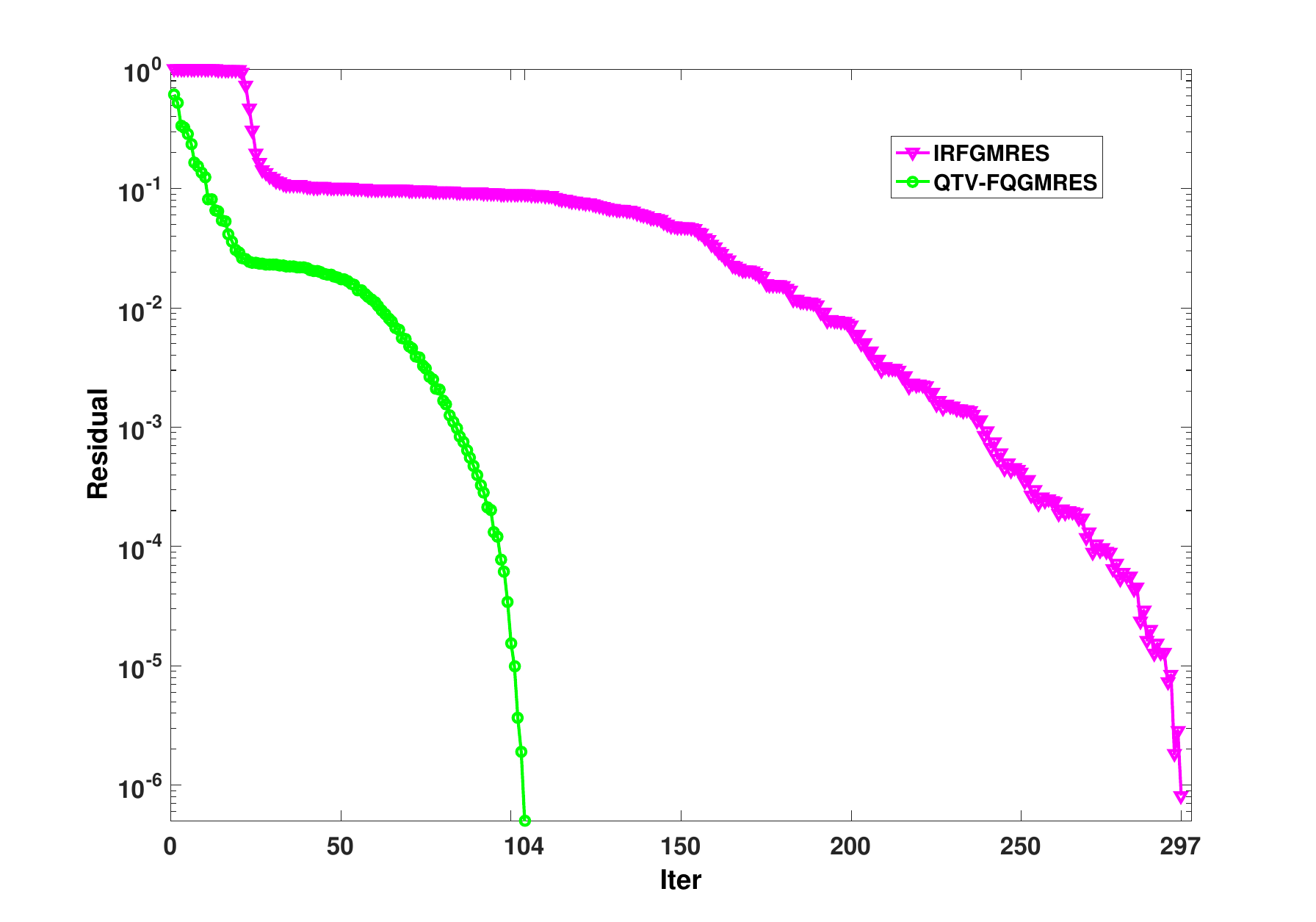}
\end{center}
 \caption{The relative residual errors of IRFGMRES and QTV-FQGMRES.}
\end{figure}

\end{example}

\begin{example}
In this example, we consider the color image pixel prediction problem. one can represent a color image with the spatial resolution of $m\times n$ pixels by an $m\times n$ pure quaternion matrix,
$\Aq =A_0+ A_1\iq + A_2\jq + A_3\kq$,
where $ A_0=0,~ A_1=(r_{ij}),~ A_2=(g_{ij}),~ A_3=(b_{ij})\in\mathbb{R}^{m\times n} $, and $ r_{ij}$, $ g_{ij}$ and $ b_{ij}$ are respectively the red, green and blue pixel values at the location $(i,j)$ in the image. For these tests, we artificially added zero-mean Gaussian noise of standard deviation 5 to a noise-free color image.
The testing color image Pepper, Mand, Traffic, Butterflies are shown in Figure \ref{fig:deblur}, whose size are $m=100;\ n=100$. Setting gauss noise=0.5.
The quality of the restored color image is indicated by the three standard criteria: PSNR, SNR, and SSIM.
The corresponding PSNR, SNR and SSIM values are in Table \ref{table3}.

We find that QTV-FQGMRES outperforms other testing methods in terms of the values. The quality of the restored images can be evaluated visually besides the values.
In Figure \ref{fig:deblur}, from left to right are the original images, observation images with added noise and blur; restored results by using IRFGMRES, SV-TV, QTV-FQGMRES.
It can be observed that the recovery effect is better with the quaternion QTV model in Figure \ref{fig:deblur}.
We observe that detailed geometry and texture can be recovered from the corrupted data by using the proposed QTV-FQGMRES model, and meanwhile, color distortion has been eliminated.
In Table \ref{table3}, we list the PSNR, SNR and SSIM values of rescovered color images by three models.
One can see that QTV-FQGMRES performs best among compared models.
Compared with the well-known SV-TV model, QTV-FQGMRES enhances PSNR, SNR and SSIM value  about 20\%, 25\% and 23\%.
\end{example}

\begin{table}[!htb]\label{table3}
\begin{lrbox}{\tablebox}
\begin{tabular}{|c||c|c|c|c|c|}
 \hline
$Images$     &Method     &PSNR&SNR &SSIM       \\  \hline\hline
\multirow{3}{*}{Pepper}
	&IRFGMRES    &   22.5596  & 16.5842 & 0.7657    \\
&SV-TV  &    27.0440    &  21.0686       & 0.8946    \\
	&QTV-FQGMRES  &   {\bf 31.6078 }  &   {\bf 25.6324  }  &{\bf 0.9261 }     \\ \hline
\multirow{3}{*}{Mand}
  &IRFGMRES   &   24.1224                      &  18.6398         & 0.8457       \\
&SV-TV    &    26.3023    &  20.8198       &  0.7576              \\
     &QTV-FQGMRES  &   {\bf  31.5988}        &  {\bf 26.1163}   &  {\bf   0.9448}   \\ \hline
\multirow{3}{*}{traffic}
&IRFGMRES   &   23.3553      &  15.5036       &  0.7105         \\
&SV-TV  &      26.2662      & 18.4145      & 0.8389             \\
&QTV-FQGMRES   &     {\bf 30.4904}              &   {\bf 22.6388}    &  {\bf   0.8945}    \\ \hline
\multirow{3}{*}{butterflies}
 &IRFGMRES    & 19.9320                   & 13.3098        &  0.7232        \\
&SV-TV  &    26.7618                 &  20.1396     &  0.8593
        \\
     &QTV-FQGMRES  &    {\bf  28.5739  }           &  {\bf  21.9517 }        &  {\bf 0.9150 }
 \\ \hline
\end{tabular}
\end{lrbox}
\centering
\caption{Numerical results of IRFGMRES, SV-TV and QTV-FQGMRES.}
\scalebox{0.85}{\usebox{\tablebox}}
\end{table}

 \begin{figure}[!h]
\centering
\includegraphics[width=0.9\textwidth]{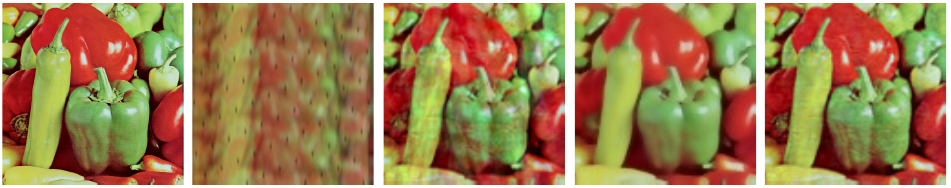}\\
\includegraphics[width=0.9\textwidth]{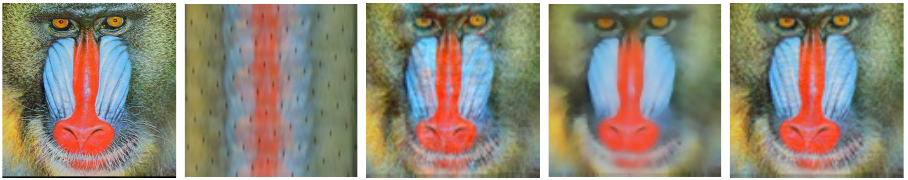}\\
\includegraphics[width=0.9\textwidth]{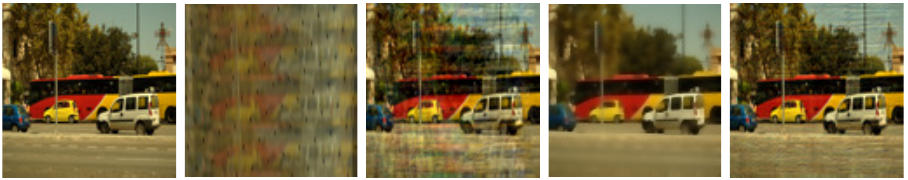}\\
\includegraphics[width=0.9\textwidth]{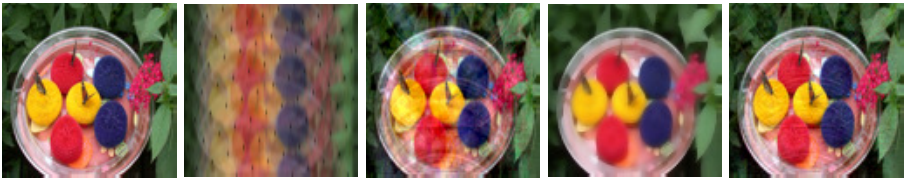}
\caption{ From left to right: orginal images; observed images; restored results by using IRFGMRES, SV-TV, QTV-FQGMRES.}
\label{fig:deblur}
\end{figure}

\section{Conclusion}\label{s:conclusion}

To solve the problem of three-dimensional signal or color image restorations,
we proposed a new quaternion total variation regularization model and developed fast and stable algorithms.
The proposed preconditioned QGMRES algorithm is a difficult problem left for future research in \cite{jnw19}.
An improved convergence theory of QGMRES is also established.
Numerical results have demonstrated that the proposed model and algorithms perform better than other compared methods
and especially, they improve image quality by 20\% in color image restoration.
In the future, we will study the structured (such as Toeplitz, Hankel, etc.) quaternion linear systems and present the related preconditioned algorithms.

\end{document}